\documentclass[12pt]{amsart}

\usepackage{framed}
\usepackage{braket}
\usepackage{amsmath,amssymb,amsthm,stmaryrd,url}
\usepackage{color}
\usepackage{bm}
\usepackage[all]{xy}
\usepackage{amscd}
\usepackage{graphicx}
\usepackage{ascmac}
\usepackage{BOONDOX-frak, mathrsfs}
\usepackage[top=30truemm,bottom=30truemm,left=25truemm,right=25truemm]{geometry}
\usepackage{mathtools}
\mathtoolsset{showonlyrefs=true}
\usepackage{tikz}
\usetikzlibrary{cd}

\makeatletter
\@addtoreset{equation}{section}

\makeatother

\newcommand{\Z}{\mathbb{Z}}
\newcommand{\N}{\mathbb{N}}
\newcommand{\Q}{\mathbb{Q}}
\newcommand{\R}{\mathbb{R}}
\newcommand{\C}{\mathbb{C}}
\newcommand{\F}{\mathbb{F}}

\DeclareMathOperator{\Imag}{Im}

\DeclareMathOperator{\ord}{ord}

\DeclareMathOperator{\GL}{GL}

\DeclareMathOperator{\Frob}{Frob}

\DeclareMathOperator{\sgn}{sgn}

\DeclareMathOperator{\gal}{Gal}
\DeclareMathOperator{\Log}{Log}
\DeclareMathOperator{\prj}{pr}

\usepackage[OT2,T1]{fontenc}
\DeclareSymbolFont{cyrletters}{OT2}{wncyr}{m}{n}
\DeclareMathSymbol{\Sha}{\mathalpha}{cyrletters}{"58}

\usepackage[
  colorlinks=false,
  pdfborder={0 0 1.8},
  linkbordercolor={1 0 0},
  citebordercolor={0 0.55 0},
  urlbordercolor={0 0 1},
  filebordercolor={0 0.55 0.55},
  linktocpage=true
]{hyperref}
\renewcommand{\C}{\mathbb{C}}

\let\oldenumerate\enumerate
\renewcommand{\enumerate}{
   \oldenumerate
   \setlength{\itemsep}{1pt}
   \setlength{\parskip}{0pt}
   \setlength{\parsep}{0pt}
}
\let\olditemize\itemize
\renewcommand{\itemize}{
   \olditemize
   \setlength{\itemsep}{1pt}
   \setlength{\parskip}{0pt}
   \setlength{\parsep}{0pt}
}

\theoremstyle{plain}
\newtheorem{thm}{Theorem}[section]
\newtheorem{lem}[thm]{Lemma}

\newtheorem{prop}[thm]{Proposition}
\newtheorem{cor}[thm]{Corollary}

\theoremstyle{definition}

\newtheorem{rem}[thm]{Remark}

\theoremstyle{definition}
\newtheorem{exa}{Example}[section]
\newtheorem{conv}{Convention}[section]

\usepackage{booktabs}
\usepackage{multirow}

\title
[sharp/flat $2$-adic Iwasawa invariants]
{Iwasawa invariants of sharp/flat 2-adic $L$-functions\\ for quadratic twists of elliptic curves}
\author{Taiga Adachi}
\address{Joint Graduate School of Mathematics for Innovation, Kyushu University,
Motooka 744, Nishi-ku, Fukuoka 819-0395, Japan}
\email{t.adachi1729@gmail.com}

 \keywords{Iwasawa theory at $p=2$, Elliptic curve, Supersingular, Quadratic twists, 2-adic $L$-function, Iwasawa invariants, Distribution}
 \subjclass[2020]{11R23 (Primary), 11F67, 11G05, 11G40}

\date{\today}
\begin{document}

\maketitle

%%%%%%%%%%%%%%%%%%%%%
\begin{abstract}
The aim of this paper is to study the variation under quadratic twists of the analytic Iwasawa invariants of Sprung's sharp/flat 2-adic $L$-functions for elliptic curves over $\Q$ with good supersingular reduction at $2$. Under the hypothesis that the $\mu$-invariant vanishes, we obtain an explicit formula for the sharp/flat $\lambda$-invariants. This formula gives a supersingular analogue of Matsuno's formula in the good ordinary case.  As an application, we show that the sharp/flat $\lambda$-invariants can be made arbitrarily large even among quadratic twists by single primes. Moreover, using the method of Hatley--Ray, we obtain an asymptotic lower bound for the number of quadratic twists with a prescribed sharp/flat 2-adic Iwasawa $\lambda$-invariant.
\end{abstract}
%%%%%%%%%%%%%%%%%%%%%

%%%%%%%%%%%%%%%%%%%%%%
\section{Introduction}\label{sec:intro}
%%%%%%%%%%%%%%%%%%%%%%
Quadratic twists of elliptic curves provide interesting variations of arithmetic invariants, including those appearing in the Birch--Swinnerton-Dyer conjecture. In this paper, we study the variation of analytic 2-adic Iwasawa invariants of elliptic curves over $\Q$ under quadratic twists. In \cite{Mat08}, Matsuno obtained a formula for the difference of analytic Iwasawa invariants of the Mazur--Swinnerton-Dyer $2$-adic $L$-functions for elliptic curves over $\Q$ with good ordinary reduction at 2 under quadratic twists. For elliptic curves with good supersingular reduction at 2, the Amice--V\'{e}lu--Vi\v{s}ik 2-adic $L$-function is not in $\Z_2\llbracket T\rrbracket\otimes_{\Z_2}\Q_2$; hence its Iwasawa invariants cannot be defined directly. In \cite{Pol03}, Pollack constructed the \emph{plus/minus} $p$-adic $L$-functions as elements of $\Z_p\llbracket T\rrbracket\otimes_{\Z_p}\Q_p$ by suitably decomposing the $p$-adic $L$-function when the Frobenius trace $a_p$ at $p$ is zero. When $p=2$ or $3$, the Frobenius trace need not vanish. Sprung constructed the \emph{sharp/flat} $p$-adic $L$-functions in \cite{Spr12} and \cite{Spr17} by generalizing Pollack's method to include the case $a_p\neq0$. In this paper, we study the variation of Iwasawa invariants of the sharp/flat 2-adic $L$-functions under quadratic twists, and give its applications.

Let $E$ be an elliptic curve defined over $\Q$ with good supersingular reduction at 2. Let $\omega:(\Z/4\Z)^{\times}\to \{\pm1\}\subset \C_2^{\times}$ be the nontrivial character. For $i\in\{0,1\}$, denote by $L_2^{\sharp/\flat}(E,\omega^i,T)\in\Z_2\llbracket T\rrbracket$ the $\omega^i$-isotypical component of the sharp/flat 2-adic $L$-function. Let $\mu^{\sharp/\flat}_2(E,\omega^i):=\mu(L_2^{\sharp/\flat}(E,\omega^i,T))$ and $\lambda^{\sharp/\flat}_2(E,\omega^i):=\lambda(L_2^{\sharp/\flat}(E,\omega^i,T))$ denote the $2$-adic Iwasawa $\mu$- and $\lambda$-invariants, respectively. For a square-free integer $D$, let $E^D$ be the quadratic twist of $E$ associated to $\Q(\sqrt{D})/\Q$. To study the variation of Iwasawa invariants under quadratic twists, a key ingredient is the following congruence relation.
\begin{thm}[Theorem \ref{Theorem:congruence_relation}]\label{Main_Thm0}
Let $E$ be an elliptic curve defined over $\Q$ with good supersingular reduction at $2$ and square-free conductor, and let $D>0$ be a square-free integer such that $D\equiv 1\bmod 4$. Then there exist $u_{D,i}(T)\in \F_2\llbracket T\rrbracket^{\times}$ and $\Delta_E(D)\in\Z_{\geq0}$ such that, for $*\in\{\sharp,\flat\}$ and a character $\omega^i:(\Z/4\Z)^{\times}\to\C_2^{\times}$,
\begin{align}
\overline{L^*_2(E^D,\omega^i,T)}\equiv u_{D,i}(T)T^{\Delta_E(D)}\overline{L^*_2(E,\omega^i,T)}
\end{align}
in $\F_2\llbracket T\rrbracket$. 
\end{thm}
\begin{rem}
The condition $D\equiv1\bmod 4$ in Theorem \ref{Main_Thm0} implies that $E^D$ also has good supersingular reduction at $2$. Moreover, the condition $D>0$ ensures that the associated quadratic character has the sign $+1$, and hence the signs of N\'{e}ron periods are not interchanged under quadratic twists.
\end{rem}
Related congruences of $p$-adic $L$-functions
were established by Greenberg--Vatsal \cite[Theorem (3.10)]{GV00}
in the ordinary case and, more recently, by
Corpuz--Lei \cite[Theorem 4.5]{CL25+} at non-ordinary primes.
Theorem \ref{Main_Thm0} gives a corresponding
congruence for sharp/flat $2$-adic $L$-functions
under quadratic twists.

The term $\Delta_E(D)$ in Theorem \ref{Main_Thm0} gives the difference between the $\lambda$-invariants under the hypothesis $\mu_2^*(E,\omega^i)=0$ for each $*\in\{\sharp,\flat\}$. Let $\ord_2$ denote the $2$-adic valuation normalized by $\ord_2(2)=1$.

\begin{thm}[Theorem \ref{Thm:Matsuno_formula}]\label{Main_Thm1}
For each rational prime $\ell\neq2$, set
\begin{align}
n_{\ell}:=\ord_2\left(\frac{\ell^2-1}{8}\right).
\end{align}
Let $E$ be an elliptic curve defined over $\Q$ with good supersingular reduction at $2$ and let $D>0$ be a square-free integer such that $D\equiv 1\bmod 4$. Fix $*\in\{\sharp,\flat\}$ and a character $\omega^i:(\Z/4\Z)^{\times}\to\C_2^{\times}$. Assume that the conductor $N_E$ of $E$ is square-free and $\mu^*_2(E,\omega^i)=0$. Then we have $\mu^*_2(E^D,\omega^i)=0$ and 
\begin{align}\label{Main_formula}
\lambda_2^*(E^D,\omega^i)=\lambda_2^*(E,\omega^i)+\sum_{\substack{\ell\mid D\\ \ell\mid N_E}}2^{n_{\ell}}+\sum_{\substack{\ell\mid D\\ \ell\nmid N_E,2\mid \sharp\tilde{E}_{\ell}(\F_{\ell})}} 2^{n_{\ell}+1}.
\end{align}
\end{thm}
Theorem \ref{Main_Thm1} is a supersingular analogue of Matsuno's formulas \cite[Theorem 5.1 and p. 420]{Mat08} for the $2^{\infty}$-Selmer group and the Mazur--Swinnerton-Dyer $2$-adic $L$-function in the ordinary case. Such a formula may be viewed as a Kida-type formula for the quadratic extension $\Q(\sqrt{D})/\Q$. Indeed, adding $\lambda_2^*(E,\omega^i)$ to both sides of \eqref{Main_formula} yields a Kida-type formula. Related Kida-type formulas for odd supersingular primes were obtained by Pollack--Weston \cite{PW06} and Pratap--Ray \cite{PR26}. There are also many related works on the Selmer-group side. See \cite{HM99}, \cite{Mat08}, \cite{Lim22}, \cite{HL19}, \cite{FM25}, and \cite{Kat24+}, for example.

As an application of Theorem \ref{Main_Thm1}, we can make the sharp/flat analytic Iwasawa $\lambda$-invariants arbitrarily large by quadratic twists as follows.
\begin{cor}[Corollary \ref{Cor:lambda_arbitrary_large}]\label{Cor:large_lambda_twist}
Let $E$ be an elliptic curve defined over $\Q$ with good supersingular reduction at $2$ and square-free conductor. Fix $*\in\{\sharp,\flat\}$ and a character $\omega^i:(\Z/4\Z)^{\times}\to\C_2^{\times}$. Assume that $\mu_2^*(E,\omega^i)=0$. Then the set
\begin{align}
\{\lambda_2^*(E^{\ell},\omega^i)\ |\ \ell\text{ is a rational prime and }\ell\equiv1\bmod 4\}
\end{align}
is infinite. More precisely, for each nonnegative integer $k$, the density of the set
\begin{align}
\{\ell\text{ is a rational prime and }\ell\equiv 1\bmod 4\ |\ \lambda_2^*(E^{\ell},\omega^i)=\lambda_2^*(E,\omega^i)+2^{k+1}\}
\end{align}
is $\frac{1}{3\cdot 2^{k+1}}$.
\end{cor}
Corollary \ref{Cor:large_lambda_twist} may be viewed as a refined supersingular analogue of \cite[Proposition 6.1]{Mat08}. Matsuno obtained arbitrarily large $\lambda$-invariants by allowing twists by products of primes. We show that arbitrarily large $\lambda$-invariants already occur among quadratic twists by a single prime, and give an explicit density estimate.

A different type of distributional estimate was recently studied by Hatley--Ray \cite{HR24+} for elliptic curves with good ordinary reduction at $2$, using Matsuno's formula. Applying their framework to the sharp/flat 2-adic $L$-functions, we obtain an asymptotic lower bound for the number of quadratic twists with a prescribed sharp/flat Iwasawa $\lambda$-invariant. 
In the following theorem, for two positive-valued functions $F(X),G(X)$ defined for all sufficiently large real numbers $X$, we write $F(X)\gg G(X)$ to mean that there exist constants $c>0$ and $X_0\in \R$ satisfying $F(X)\geq cG(X)$ for all $X\geq X_0$.
\begin{thm}[Theorem \ref{Thm:distribution_lambda}]
Let $E$ be an elliptic curve defined over $\Q$ with good supersingular reduction at $2$ and square-free conductor. Fix $*\in\{\sharp,\flat\}$ and a character $\omega^i:(\Z/4\Z)^{\times}\to\C_2^{\times}$. Let $N^*$ be a nonnegative integer such that $N^*\geq \lambda_2^*(E,\omega^i)$ and $\lambda_2^*(E,\omega^i)\equiv N^*\bmod 2$. Assume that $\mu^*_2(E,\omega^i)=0$. Let
\begin{align}
m_{E,N^*}^*(X):=\sharp\{D\leq X\ |\ D\in\Z_{>0}\text{ is square-free, }D\equiv1\bmod 4,\text{ and }\lambda_2^*(E^D,\omega^i)=N^*\}.
\end{align}
Then we have
\begin{align}
m_{E,N^*}^*(X)\gg \frac{X}{(\log X)^{\frac{2}{3}}}.
\end{align}
\end{thm}
A related distributional result at $p=3$ was obtained by Burungale--Skinner \cite{BS23} for elliptic curves for which $3$ is an Eisenstein prime of good reduction. Controlling $3$-adic $\lambda$-invariants under quadratic twists, they show that a positive proportion of quadratic twists have the Mordell--Weil rank one and the non-degenerate $3$-adic regulator.

\subsection*{Organization of the paper}In \S \ref{sec:sharp/flat_L}, we review the sharp/flat 2-adic $L$-function constructed by Sprung in \cite{Spr12} and \cite{Spr17}.
In \S \ref{sec:quad_analytic_Iwasawa}, we establish
a Greenberg--Vatsal-type congruence \cite{GV00}
for sharp/flat $2$-adic $L$-functions via finite-layer
congruences of Mazur--Tate elements, and use it to
obtain a Matsuno-type formula for Iwasawa invariants
under quadratic twists. Our method also draws on
arguments of Sinnott \cite{Sin84} and Matsuno
\cite{Mat00} in their proofs of Kida-type formulas
for $p$-adic $L$-functions.
In \S \ref{sec:distribution}, as an application of the Matsuno-type formula proved in \S\ref{sec:quad_analytic_Iwasawa}, we obtain an asymptotic lower bound for the number of quadratic twists with a prescribed sharp/flat Iwasawa $\lambda$-invariant, following the method of Hatley--Ray \cite{HR24+}. In \S \ref{sec:example}, we give some examples of our results using Pollack's SageMath code \cite{PolCode}.

\section*{Acknowledgments}

The author would like to express sincere gratitude to Florian Sprung for reading an earlier version of this manuscript with much interest, for pointing out several mathematical mistakes and suggesting various improvements throughout the manuscript. In particular, the author is grateful to him for observing that, with slight modifications, the corresponding results in the earlier version could be strengthened to give Lemma \ref{Lem:linear_disjoint} and Corollary \ref{Cor:large_lambda_twist}.
The author is also deeply grateful to Shinichi Kobayashi for helpful comments, encouragement, and fruitful discussions. He is very grateful to Li-Tong Deng for explaining how this work is related to Deng's announced work \cite{Den} and to the joint work \cite{DL26+} with Yong-Xiong Li. The author would like to thank Ashay Burungale, Chan-Ho Kim, Takenori Kataoka, and Mahiro Atsuta for their encouragement, discussions, and insightful comments. Parts of this work were carried out during the author’s stay at the University of Ottawa in October 2024 and at Seoul National University in December 2025--January 2026, and the author is grateful to his host professors, Antonio Lei and Dohyeong Kim, for their hospitality and helpful comments during his visits. The author would like to thank Keiichiro Nomoto and Ryota Shii; without the inspiration from the author's joint work \cite{ANS26} with them, this research would not have been undertaken.

\section{The sharp/flat 2-adic $L$-function}\label{sec:sharp/flat_L}
In this section, we review the construction of the sharp/flat 2-adic $L$-function for weight-two modular forms by Sprung \cite{Spr17}.

Fix two algebraic closures $\overline{\Q}$ and $\overline{\Q_2}$ of $\Q$ and $\Q_2$, respectively. We also fix an embedding $\overline{\Q}\hookrightarrow \C_2$. Let $\ord_2$ be the associated 2-adic valuation such that $\ord_2(2)=1$ and let $|\cdot|_2$ be the corresponding absolute value. We fix a system $(\zeta_S)_{S\in\N}$ in $\overline{\Q}$, namely, $\zeta_S$ is a primitive $S$-th root of unity, and $(\zeta_S)^{S/S'}=\zeta_{S'}$ for $S'\mid S$.

%%%%%%%%%%%%%%%%%%%%%%
\subsection{Modular symbols and the 2-adic $L$-function}\label{ss:2-adic_L}
%%%%%%%%%%%%%%%%%%%%%%

Let $f=\sum_{n\geq1}a_n(f)q^n$ be a normalized weight-two newform of odd level $N$ with trivial nebentypus with rational Fourier coefficients. We sometimes write $a_n$ for $a_n(f)$. Assume that $\ord_2(a_2)>0$. 
\begin{conv}\label{Convention:period_normalization}
For an elliptic curve $A$ over $\Q$, define the N\'{e}ron periods $\Omega_A^{\pm}$ by
\begin{align}
\Omega_A^{\pm}:=\kappa(A)\int_{\delta_A^{\pm}}\omega_A,
\end{align}
where
\begin{align}
\kappa(A):=\begin{cases}
1&(\text{if }A(\R)\text{ is connected}),\\
2&(\text{otherwise}).
\end{cases}
\end{align}
Here, $\omega_A$ is the N\'{e}ron differential of $A$ and $\delta_A^{\pm}$ is a generator of $H_1(A(\C),\Z)^{\pm}$. 
Let $\pi_A:X_0(N)\to A$ be the modular parametrization and let $A_0$ be the optimal elliptic curve in the $\Q$-isogeny class of $A$. Then there exists a constant $c(A_0)$ such that
\begin{align}\label{eq:differential_form_pullback}
\pi_{A_0}^*\omega_{A_0}=c(A_0)\cdot 2\pi if_{A_0}(z)dz.
\end{align}
The constant $c(A_0)$ is called the \emph{Manin constant} and it is known to be an integer (see \cite[Proposition 2]{Edi91}). Moreover, $c(A_0)$ is conjectured to be $1$ (see \cite[\S5]{Man72}). Throughout this paper, we choose the \emph{Shimura periods} $\Omega_f^{\pm}$ constructed in \cite{Shi77} and normalize them by 
\begin{align}
\Omega_{f}^{\pm}=c(A_f)^{-1}\Omega_{A_f}^{\pm},
\end{align}
where $A_f$ is the optimal quotient elliptic curve associated to $f$. 
\end{conv}
We define the \emph{periods} of $f$ by
\begin{align}
\phi(f,r):=2\pi i\int_{i\infty}^{r}f(z)dz
\end{align}
for $r\in \Q$ and we put 
\begin{align}
\eta(f,r)^{\pm}:=\frac{\phi(f,r)\pm \phi(f,-r)}{2}.
\end{align}

We first recall a basic property of the residual Galois representation at $2$, which will be used in the proof of the integrality of modular symbols.

\begin{prop}\label{Prop:mod_2_surjective}
Let $A$ be an elliptic curve defined over $\Q$ with good supersingular reduction at $2$. Then we have $A(\Q)[2]=\{O\}$ and the residual representation
\[
\bar{\rho}_{A,2}:G_{\Q}\to\GL(A[2])
\]
is surjective.
\end{prop}

\begin{proof}
By a result of Serre \cite[Proposition 12]{Ser72}, after choosing a suitable $\F_2$-basis of $A[2]$, we see that $\bar{\rho}_{A,2}(G_{\Q_2})$ is the normalizer $N_{\GL_2(\F_2)}(\F_4^{\times})$ of $\F_4^{\times}$, where $\F_4^{\times}$ acts on $\F_4$ by multiplication. Since $\F_4^{\times}$ acts transitively on the three nonzero elements of $A[2]$, we have
\begin{align}
A(\Q_2)[2]=A[2]^{G_{\Q_2}}=\{O\}.
\end{align}
Hence $A(\Q)[2]=\{O\}$.

Next, let $\sigma\in\gal(\F_4/\F_2)$ be the Frobenius element. Then we have $\sigma a\sigma^{-1}=a^2$ for $a\in\F_4^{\times}$ and
\begin{align}
\F_4^{\times}\rtimes\gal(\F_4/\F_2)
\subseteq N_{\GL_2(\F_2)}(\F_4^{\times}).
\end{align}
Since the order of $\F_4^{\times}\rtimes\gal(\F_4/\F_2)$ is $6$, which equals the order of $\GL_2(\F_2)$, we obtain
\begin{align}
N_{\GL_2(\F_2)}(\F_4^{\times})=\GL_2(\F_2).
\end{align}
Thus $\bar{\rho}_{A,2}$ is surjective, as required.
\end{proof}

\begin{thm}[cf. {\cite{Man72}}, {\cite[Theorem 1]{Shi77}}, {\cite[Theorem 3.5.4]{GS94}}, and {\cite[Theorem 5.15 (2)]{Spr17}}]\label{Thm:periods}
Assume that $N$ is square-free. Then, for all $r\in \Q$, we see that
\begin{align}
[r]_f^{\pm}:=\frac{\eta(f,r)^{\pm}}{\Omega_f^{\pm}}\in \frac{1}{2}\Z_2.
\end{align}
\end{thm}
\begin{proof}
The integrality of modular symbols has been studied by Manin \cite{Man72}, Greenberg--Stevens \cite[Theorem 3.5.4]{GS94}, and Sprung \cite[Theorem 5.15 (2)]{Spr17}, for example. In our setting, with the periods normalized in terms of N\'{e}ron periods as in Convention \ref{Convention:period_normalization}, we prove that $2[r]_f^\pm\in\Z_2$ for every $r\in\Q$ under the square-free conductor assumption.

Since $N$ is square-free, \cite[Corollary 2.3 and Theorem 2.17]{Yoo23} imply that all cusps of $X_0(N)$ are $\Q$-rational. We normalize the modular parametrization $\pi_{A_f}$ by $\pi_{A_f}(\infty)=O$. Let $r\in\Q$. By the Manin--Drinfeld theorem, we see that $\pi_{A_f}(r)\in A_f(\Q)_{\mathrm{tors}}$. Since $A_f$ has good supersingular reduction at $2$, it follows that $A_f(\Q)[2]=\{O\}$ by Proposition \ref{Prop:mod_2_surjective}. Hence the order $m_r$ of $\pi_{A_f}(r)$ in $A_f(\Q)$ is odd. Let $\sigma_r$ denote the image in $X_0(N)(\C)$ of a path from $i\infty$ to $r$ in $\mathbb{H}\cup\mathbb{P}^1(\Q)$. Let $\gamma_r\in H_1(A_f(\C),\Z)$ be the homology class of the closed path $[m_r]\circ \pi_{A_f}\circ \sigma_r$. By \eqref{eq:differential_form_pullback} and $[m_r]^*\omega_{A_f}=m_r\omega_{A_f}$, we have
\begin{align}
m_rc(A_f)\phi(f,r)&=m_r\int_{\sigma_r}\pi_{A_f}^*\omega_{A_f}=\int_{\pi_{A_f}\circ \sigma_r}[m_r]^*\omega_{A_f}=\int_{\gamma_r}\omega_{A_f}.
\end{align}
Let $c_{A_f}:A_f(\C)\to A_f(\C)$ denote complex conjugation and $(c_{A_f})_*:H_1(A_f(\C),\Z)\to H_1(A_f(\C),\Z)$ be the morphism induced by $c_{A_f}$. Let $c_X$ be the complex conjugation on $X_0(N)(\C)$. Under the standard identification $X_0(N)(\C)\cong \Gamma_0(N)\backslash (\mathbb{H}\cup \mathbb{P}^1(\Q))$, the map $c_X$ is induced by $z\mapsto -\overline{z}$. Thus, $c_X\circ \sigma_r$ is the image of a path from $i\infty$ to $-r$, which can be used to compute $\phi(f,-r)$. Since $\pi_{A_f}$ is defined over $\Q$, we also have $\pi_{A_f}\circ c_X=c_{A_f}\circ \pi_{A_f}$. Hence,
\begin{align}
m_rc(A_f)\phi(f,-r)&=\int_{[m_r]\circ\pi_{A_f}\circ c_X\sigma_r}\omega_{A_f}=\int_{c_{A_f}\circ [m_r]\circ \pi_{A_f}\circ \sigma_r}\omega_{A_f}=\int_{(c_{A_f})_*\gamma_r}\omega_{A_f}.
\end{align}
Write $(1\pm (c_{A_f})_*)\gamma_r=k_{\pm}\delta_{A_f}^{\pm}$ for some integers $k_{\pm}$. Hence, we obtain
\begin{align}
2[r]_f^{\pm}=\frac{k_{\pm}}{m_r\kappa(A_f)}.
\end{align}
If $\kappa(A_f)=1$, the assertion follows immediately since $m_r$ is odd. It remains to consider the case $\kappa(A_f)=2$. In this case, $A_f[2]\subset A_f(\R)$. The complex uniformization induces an isomorphism
\begin{align}
H_1(A_f(\C),\Z)/2H_1(A_f(\C),\Z)\cong A_f[2]
\end{align}
compatible with complex conjugation. Hence $(c_{A_f})_*$ acts trivially on $H_1(A_f(\C),\Z)/2H_1(A_f(\C),\Z)$, so
\begin{align}
(1\pm (c_{A_f})_*)H_1(A_f(\C),\Z)\subseteq 2 H_1(A_f(\C),\Z)\cap H_1(A_f(\C),\Z)^{\pm}=2H_1(A_f(\C),\Z)^{\pm}.
\end{align}
Thus $k_{\pm}$ is even. Since $m_r$ is odd and $\kappa(A_f)=2$, we conclude that $2[r]_f^{\pm}\in \Z_2$, as required.
\end{proof}
The symbols $[r]_f^{\pm}$ in Theorem \ref{Thm:periods} are called \emph{modular symbols}.

We put $\Z_{2,M}:=\Z_2\times (\Z/M\Z)$ and $\Z_{2,M}^{\times}:=\Z_2^{\times}\times (\Z/M\Z)^{\times}$ for some odd integer $M$. Denote by $C^0(\Z_{2,M}^{\times})$ the $\C_2$-valued step functions on $\Z_{2,M}^{\times}$. Let $\alpha,\beta$ be the roots of the Hecke polynomial $X^2-a_2(f)X+2$ at $2$. In this case, we see that $\alpha\neq\beta$ and $\ord_2(\alpha)= \ord_2(\beta)=1/2$. For $\xi\in\{\alpha,\beta\}$, we define two linear maps from $C^0(\Z_{2,M}^{\times})$ to $\C_2$ by
\begin{align}
\mu_{f,\xi,M}^{\pm}(a+2^nM\Z_{2,M}):=\frac{1}{\xi^{n+1}}\left(\left[\frac{a}{2^nM}\right]_f^{\pm},\left[\frac{a}{2^{n-1}M}\right]_f^{\pm}\right)\begin{pmatrix}\xi\\ -1\end{pmatrix}
\end{align}
for each integer $n\geq1$ and each integer $a$ coprime to $2M$. 
\begin{thm}[cf. {\cite[Lemma 1.6]{Vis76}}]
For $\xi\in\{\alpha,\beta\}$, $\mu_{f,\xi,M}^{\pm}$ extend to two linear maps on the space of all locally analytic functions on $\Z_{2,M}^{\times}$.
\end{thm}
\begin{prop}
For $\xi\in\{\alpha,\beta\}$, $\mu_{f,\xi,M}^{\pm}$ is a $1/2$-admissible measure.
\end{prop}
\begin{proof}
Additivity follows from the argument in \cite[Section 10]{MTT86}, and the admissibility follows from \cite[Lemma 3.8]{Vis76}.
\end{proof}
Since each continuous character $\psi:\Z_{2,M}^{\times}\to\C_2^{\times}$ is locally analytic, we can define, for $\xi\in\{\alpha,\beta\}$, an element
\begin{align}
L_{2,M}(f,\xi,\psi):=\mu_{f,\xi,M}^{\sgn(\psi)}(\psi),
\end{align}
where $\sgn(\psi)$ is the sign $\psi(-1)$ of $\psi$. 
The function $L_{2,M}(f,\xi,\psi)$ has the following interpolation property for each $\xi\in\{\alpha,\beta\}$.
\begin{thm}[Amice--V\'{e}lu {\cite{AV75}}, Vi\v{s}ik {\cite{Vis76}}; cf. {\cite[Proposition 2.11]{Pol03}}]
Let $\psi:\Z_{2,M}^{\times}\to\C_2^{\times}$ be a finite-order character whose associated primitive Dirichlet character, also denoted by $\psi$, has conductor $2^{\nu}M$. Then, for $\xi\in\{\alpha,\beta\}$, $L_{2,M}(f,\xi,\psi)$ satisfies the following interpolation formula
\begin{align}\label{eq:interpolation_formula}
L_{2,M}(f,\xi,\psi)=\frac{1}{\xi^{\nu}}\left(1-\frac{\psi^{-1}(2)}{\xi}\right)\left(1-\frac{\psi(2)}{\xi}\right)\frac{2^{\nu}M}{\tau(\psi^{-1})}\frac{L(f\otimes \psi^{-1},1)}{\Omega_f^{\sgn(\psi)}},
\end{align}
where $L(f\otimes\psi^{-1},s)=\sum_n(a_n\psi^{-1}(n)/n^s)$ is the complex $L$-function of $f\otimes\psi^{-1}$, $\tau(\psi^{-1})$ is the Gau{\ss} sum, and $f\otimes \psi^{-1}$ is the twist of $f$ by the finite character $\psi^{-1}$.
\end{thm}
Using the decomposition $\Z_{2,M}^{\times}\cong (\Z/4M\Z)^{\times}\times (1+4\Z_2)$, $\psi$ can be uniquely written as the product $\psi=\psi_t\psi_u$ of the \emph{tame} character $\psi_t$ factoring through $(\Z/4M\Z)^{\times}$ and the \emph{wild} character $\psi_u$ factoring through $1+4\Z_2$. Here, $\psi_u:1+4\Z_2\to\C_2^{\times}$ is the character sending the topological generator $5$ of $1+4\Z_2$ to $u$ for some $u\in\C_2$ with $|u-1|_2<1$. When $M=1$, $\psi_t$ has the form $\omega^i$ with $i\in\{0,1\}$, where $\omega$ is the natural inclusion $\omega:(\Z/4\Z)^{\times}\hookrightarrow \C_2^{\times}$. See also \cite[p.528]{Pol03} for details.
\begin{thm}[Amice--V\'{e}lu {\cite{AV75}}, Vi\v{s}ik {\cite{Vis76}}]\label{Thm:2-adic_L_analytic}
For $\xi\in\{\alpha,\beta\}$ and a fixed tame character $\psi_t:(\Z/4M\Z)^{\times}\to \C_2^{\times}$, $L_{2,M}(f,\xi,\psi_t\psi_u)$ is analytic on the open unit disc $\{u\in\C_2\ |\ |u-1|_2<1\}$.
\end{thm}

From Theorem \ref{Thm:2-adic_L_analytic}, we obtain the power series expansion about $u=1$. Changing the variable $T=u-1$, we denote $L_{2,M}(f,\xi,\psi_t\psi_u)$ by $L_{2,M}(f,\xi,\psi_t,T)$ for $\xi\in\{\alpha,\beta\}$. It is known (cf. \cite[p.531]{Pol03}) that
\begin{align}
L_{2,M}(f,\xi,\psi_t,T)\in \Q_2(\xi,\psi_t)\llbracket T\rrbracket
\end{align}
for each $\xi\in\{\alpha,\beta\}$. 
%In the rest of the paper, for a fixed tame character $\omega^i:\Delta_1\to\C_2^{\times}$, we sometimes use the same symbol $\omega^i$ for the extended character on $\Delta_M$ via the natural projection $\Delta_M\twoheadrightarrow \Delta_1$ for some odd integer $M$. 
For any odd positive integer $M$, put $\Delta_M:=(\Z/4M\Z)^\times$ and $\Delta:=\Delta_1$. In the rest of the paper, for a fixed tame character $\omega^i:\Delta_1\to\C_2^\times$, we sometimes use the same symbol $\omega^i$ for the extended character on $\Delta_M$ via the natural projection $\Delta_M\twoheadrightarrow\Delta_1$.

%%%%%%%%%%%%%%%%%%%%%%
\subsection{Mazur--Tate elements}\label{ss:MT}
%%%%%%%%%%%%%%%%%%%%%%

For an integer $n\geq0$ and an odd positive integer $M$, we put $\mathcal{G}_{n,M}:=\gal(\Q(\mu_{2^nM})/\Q)$. 
Let $\Q_n:=\Q(\mu_{2^{n+2}})^+$ be the maximal real subfield of $\Q(\mu_{2^{n+2}})$, and put $\Gamma_n:=\gal(\Q_n/\Q)\cong\Z/2^n\Z$. The group $\mathcal{G}_{n,M}$ admits a decomposition
\begin{align}
\mathcal{G}_{n,M}\cong \Delta_M\times \Gamma_{n-2},
\end{align}
for $n\geq2$. We put $\Q_\infty:=\bigcup_{n\geq0}\Q_n$, $\Gamma:=\gal(\Q_\infty/\Q)\cong\varprojlim_n\Gamma_n\cong\Z_2$, and $\Lambda:=\Z_2\llbracket\Gamma\rrbracket$. Then we have the isomorphism $\Lambda\cong \Z_2\llbracket T\rrbracket$, which sends a topological generator $5$ of $\Gamma$ to $1+T$. This isomorphism induces the isomorphism $\Lambda_n:=\Z_2[\Gamma_n]\cong \Z_2\llbracket T\rrbracket/((1+T)^{2^n}-1)$ for each $n\geq0$.

We define the \emph{Mazur--Tate element} by
\begin{align}
\theta_{n,M}^{\pm}(f):=\sum_{a\in(\Z/2^nM)^{\times}}\left[\frac{a}{2^nM}\right]_f^{\pm}\sigma_a\in \Z_2[1/2][\mathcal{G}_{n,M}]
\end{align}
for $n\geq0$ and an odd integer $M$. Here, $\sigma_a$ is the element of $\mathcal{G}_{n,M}$ such that $\sigma_a(\zeta_{2^nM})=\zeta_{2^nM}^a$. For $n\geq2$ and a fixed tame character $\omega^i$, we define the $\omega^i$\emph{-isotypical component} by
\begin{align}
\varepsilon_{\omega^i}\theta_{n,M}^{\sgn(\omega^i)}(f)\in \Q_2[\mathcal{G}_{n,M}],
\end{align}
where $\varepsilon_{\omega^i}$ is the idempotent
\begin{align}
\varepsilon_{\omega^i}=\frac{1}{\sharp\Delta_M}\sum_{\tau\in\Delta_M}\omega^i(\tau)\tau^{-1}.
\end{align}
Using the natural isomorphisms
\begin{align}
\varepsilon_{\omega^i}\Q_2[\mathcal{G}_{n+2,M}]\cong (\varepsilon_{\omega^i}\Q_2[\Delta_M])\otimes_{\Q_2} \Q_2[\Gamma_n]\cong \Q_2[\Gamma_n]\cong \Q_2[T]/((1+T)^{2^n}-1),
\end{align}
and regarding $\omega^i$ as a map on $(\Z/2^{n+2}M\Z)^{\times}$ via the natural projection $(\Z/2^{n+2}M\Z)^{\times}\twoheadrightarrow \Delta_M$, $\varepsilon_{\omega^i}\theta_{n+2,M}^{\sgn(\omega^i)}(f)$ can be identified with
\begin{align}
\theta_{n,M}(f,\omega^i,T):=\sum_{a\in (\Z/2^{n+2}M\Z)^{\times}}\left[\frac{a}{2^{n+2}M}\right]_f^{\sgn(\omega^i)}\omega^i(a)(1+T)^{t_n(a)}\in \Q_2[T]/((1+T)^{2^n}-1).
\end{align}
Here, $t_n(a)$ is the unique element of $\Z/2^n\Z$ such that
\begin{align}
5^{t_n(a)}\equiv\langle a\rangle \bmod 2^{n+2},
\end{align}
where $\langle a\rangle$ is the second projection of the decomposition
\begin{align}
(\Z/2^{n+2}M\Z)^{\times}\cong (\Z/4M\Z)^{\times}\times (1+4\Z_2)/(1+2^{n+2}\Z_2).
\end{align}
Similarly, we define the map $t:\Z_{2,M}^{\times}\to\Z_2$ by
\begin{align}
\langle a\rangle=5^{t(a)},
\end{align}
where $\langle a\rangle$ is the image of the projection $\Z_{2,M}^{\times}\to 1+4\Z_2$. 
More generally, for any odd integer $a$, we define $t_n(a)$ and $t(a)$ by taking $M=1$ in the above definitions. 

\begin{prop}[cf. {\cite[Corollary 5.6]{Spr17}}]\label{Prop:MT_integral}
Assume that $N$ is square-free. Then, for each $n\geq0$, we have
\begin{align}
\theta_{n,M}(f,\omega^i,T)\in\Lambda_n.
\end{align}
\end{prop}
\begin{proof}
Since the involution $a\mapsto -a$ on $(\Z/2^{n+2}M\Z)^{\times}$ has no fixed points, by Theorem \ref{Thm:periods}, we see that
\begin{align}
&\theta_{n,M}(f,\omega^i,T)\\
&=\sum_{\substack{a\in (\Z/2^{n+2}M\Z)^{\times}\\ 1\leq a<2^{n+1}M}}\left(\left[\frac{a}{2^{n+2}M}\right]_f^{\sgn(\omega^i)}\omega^i(a)(1+T)^{t_n(a)}+\left[\frac{-a}{2^{n+2}M}\right]_f^{\sgn(\omega^i)}\omega^i(-a)(1+T)^{t_n(-a)}\right)\\
&=\sum_{\substack{a\in (\Z/2^{n+2}M\Z)^{\times}\\ 1\leq a<2^{n+1}M}}(1+\sgn(\omega^i)^2)\left[\frac{a}{2^{n+2}M}\right]_f^{\sgn(\omega^i)}\omega^i(a)(1+T)^{t_n(a)}\\
&=\sum_{\substack{a\in (\Z/2^{n+2}M\Z)^{\times}\\ 1\leq a<2^{n+1}M}}\left(2\left[\frac{a}{2^{n+2}M}\right]_f^{\sgn(\omega^i)}\right)\omega^i(a)(1+T)^{t_n(a)}\in \Lambda_n.
\end{align}
\end{proof}
Let $\pi_{n/n-1}:\Lambda_n\to \Lambda_{n-1}$ be the natural projection. We define the map $\nu_{n-1/n}:\Lambda_{n-1}\to \Lambda_n$ by
\begin{align}
\nu_{n-1/n}(\sigma):=\sum_{\substack{\tau\in\Gamma_n\\ \pi_{n/n-1}(\tau)=\sigma}}\tau
\end{align}
for $\sigma\in \Gamma_{n-1}$. 

%%%%%%%%%%%%%%%%%%%%%%
\subsection{The logarithmic matrix and the sharp/flat $2$-adic $L$-function}\label{ss:2-adic_L-function}
For each positive integer $j$, we put
\begin{align}
C_j:=C_j(f):=\begin{pmatrix}a_2(f)& 1\\ -\Phi_{2^j}(1+T)&0\end{pmatrix},\quad C:=C(f):=\begin{pmatrix}a_2(f)&1\\ -2&0\end{pmatrix},
\end{align}
where $\Phi_{2^j}(X)$ is the $2^j$-th cyclotomic polynomial.

Although Sprung formulates his construction only in the case $M=1$, the same argument applies verbatim to any odd integer $M$. Indeed, the Mazur--Tate elements with $M$ still form a queue sequence in the sense of Sprung, that is, the sequence $(\theta_{n,M}(f,\omega^i,T))_{n\geq 0}\in (\Lambda_n)_{n\geq0}$ satisfies
\begin{align}
\pi_{n/n-1}(\theta_{n,M}(f,\omega^i,T))=a_2(f)\theta_{n-1,M}(f,\omega^i,T)-\nu_{n-2/n-1}(\theta_{n-2,M}(f,\omega^i,T))
\end{align} 
for each $n\geq2$ (cf. \cite[(4.2)]{MTT86}). Moreover, the matrix calculation in Sprung's construction involves only the matrices $C_j$, which depend only on $f$ and not on the auxiliary integer $M$. Therefore, Sprung's construction remains valid in this more general setting.
\begin{prop}[{\cite[Corollary 5.6]{Spr17}}]\label{Prop:finite_exist}
For each $n\geq1$, there exists a vector 
\begin{align}
\vec{L}_{2,n,M}^{\omega^i}=(L_{2,n,M}^{\sharp}(f,\omega^i,T),L_{2,n,M}^{\flat}(f,\omega^i,T))\in\Lambda_n^{\oplus2}
\end{align}
such that
\begin{align}\label{eq:exist_finite_sharp/flat}
&(\theta_{n,M}(f,\omega^i,T),\nu_{n-1/n}(\theta_{n-1,M}(f,\omega^i,T)))R_n(\alpha,\beta)\\
&=\vec{L}_{2,n,M}^{\omega^i}C_1\cdots C_nC^{-(n+3)}\begin{pmatrix}-1&-1\\ \beta&\alpha\end{pmatrix},
\end{align}
where
\begin{align}
R_n(\alpha,\beta):=\begin{pmatrix}\alpha^{-(n+2)}&\beta^{-(n+2)}\\ -\alpha^{-(n+3)}&-\beta^{-(n+3)}\end{pmatrix}.
\end{align}
\end{prop}
\begin{prop}[{\cite[Lemma 4.4]{Spr17}}]\label{Prop:Sprung_Matrix}
All entries of the following matrix
\begin{align}
C_1\cdots C_nC^{-(n+3)}\begin{pmatrix}-1&-1\\ \beta&\alpha\end{pmatrix}
\end{align}
converge on the open unit disc as $n\to \infty$.
\end{prop}
Denote by $\Log(f)$ the limit in Proposition \ref{Prop:Sprung_Matrix}. From \cite[Remark 4.5]{Spr17},
\begin{align}
\det\Log(f)=\frac{\log_2(1+T)}{T}\times \frac{\beta-\alpha}{8}.
\end{align}
Although the vector $\vec{L}^{\omega^i}_{2,n,M}$ in Proposition \ref{Prop:finite_exist} is not unique at finite level, the ambiguity is measured, as in \cite[Definition 5.9]{Spr17}, by
\begin{align}
\mathcal{M}_n(f):=\ker\left(\times C_1\cdots C_nC^{-(n+3)}\begin{pmatrix}-1&-1\\ \beta&\alpha\end{pmatrix}\right)\subset\Lambda_n^{\oplus2}.
\end{align}
This module is independent of the auxiliary odd integer $M$. By \cite[Proposition 5.10]{Spr17}, in the supersingular case considered here we have
\begin{align}
\mathcal{M}(f):=\varprojlim_n\mathcal{M}_n(f)=\{(0,0)\}.
\end{align}
Therefore the ambiguity of the finite-level choices disappears after taking the projective limit. Consequently, Sprung's limiting construction gives a well-defined element of $\Lambda^{\oplus2}$, and we denote it by
\begin{align}
(L_{2,M}^{\sharp}(f,\omega^i,T),L_{2,M}^{\flat}(f,\omega^i,T))\in\Lambda^{\oplus2}.
\end{align}

\begin{rem}
Throughout this paper, we define the sharp/flat $2$-adic $L$-functions using the Shimura periods $\Omega_f^{\pm}$ fixed in Convention \ref{Convention:period_normalization}. Sprung also considers the normalization by N\'{e}ron periods \cite[Remark 5.12 and Convention 5.14]{Spr17}. For non-optimal elliptic curves, however, \cite[Corollary 5.17]{Spr17} guarantees only that the resulting functions lie in $\Lambda\otimes_{\Z_2}\Q_2$, not necessarily in $\Lambda$. We compare the two normalizations in Proposition \ref{Prop:elliptic_modular_2-adic_unit}.
\end{rem}

\begin{thm}[{\cite[Theorem 2.14]{Spr17}}]\label{Thm:exist_sharp_flat}
The elements $L_{2,M}^{\sharp/\flat}(f,\omega^i,T)\in\Lambda$ satisfy
\begin{align}
(L_{2,M}(f,\alpha,\omega^i,T),L_{2,M}(f,\beta,\omega^i,T))=(L_{2,M}^{\sharp}(f,\omega^i,T),L_{2,M}^{\flat}(f,\omega^i,T))\Log(f).
\end{align}
\end{thm}
\begin{prop}\label{Euler_factor_difference}
For an odd rational prime $\ell$, we put 
\begin{align}
h_{\ell}(f,\omega^i,T):=a_{\ell}(f)-\omega^i(\ell)(1+T)^{t(\ell)}-\varepsilon(\ell)\omega^{-i}(\ell)(1+T)^{-t(\ell)},
\end{align}
where $\varepsilon$ is defined by
\begin{align}
\varepsilon(\ell):=\begin{cases}
1&(\ell\nmid N),\\
0&(\ell\mid N).
\end{cases}
\end{align}
Suppose that $M$ is square-free. Then we have, for a rational prime $\ell\mid M$, 
\begin{align}\label{eq:prime_factor_reduce}
L^{\sharp/\flat}_{2,M}(f,\omega^i,T)=h_{\ell}(f,\omega^i,T)L^{\sharp/\flat}_{2,M/\ell}(f,\omega^i,T).
\end{align}
\end{prop}
\begin{proof}
For a rational prime $\ell\mid M$, using calculations similar to those in \cite[Lemmas 2.2 and 3.3]{Mat00}, we have
\begin{align}
(L_{2,M}(f,\alpha,\omega^i,T),L_{2,M}(f,\beta,\omega^i,T))=h_{\ell}(f,\omega^i,T)(L_{2,M/\ell}(f,\alpha,\omega^i,T),L_{2,M/\ell}(f,\beta,\omega^i,T)).
\end{align}
By Theorem \ref{Thm:exist_sharp_flat} and $\det\Log(f)\neq0$, we obtain
\begin{align}
(L_{2,M}^{\sharp}(f,\omega^i,T),L_{2,M}^{\flat}(f,\omega^i,T))\Log(f)=h_{\ell}(f,\omega^i,T)(L_{2,M/\ell}^{\sharp}(f,\omega^i,T),L_{2,M/\ell}^{\flat}(f,\omega^i,T))\Log(f).
\end{align}
This proves the proposition.
\end{proof}
%%%%%%%%%%%%%%%%%%%%%%

%%%%%%%%%%%%%%%%%%%%%%
\section{Analytic sharp/flat Iwasawa invariants for quadratic twists}
\label{sec:quad_analytic_Iwasawa}
%%%%%%%%%%%%%%%%%%%%%%
In this section, we study the variation of Mazur--Tate elements under quadratic twists and obtain the Matsuno-type formula. Our proof consists of two parts: establishing a finite-layer congruence relation for Mazur--Tate elements and computing the difference between the Iwasawa invariants of the primitive and imprimitive sharp/flat $2$-adic $L$-functions.

\begin{rem}
To prove Theorem \ref{Main_Thm1}, we adopt an approach based on the study of the $2$-adic properties of Mazur--Tate elements and finite-layer congruences. A related recent work is the paper of Deng--Li \cite{DL26+}, in which Deng announces the work \cite{Den} on the variation of analytic Iwasawa invariants under quadratic twists. Deng kindly explained to the author the relationship between his announced work and the present paper. To the best of the author's knowledge, the two settings are different and neither contains the other.
\end{rem}
%%%%%%%%%%%%%%%%%%%%%%%
\subsection{Shimura periods for quadratic twists}
\label{ss:Period_twists}
%%%%%%%%%%%%%%%%%%%%%%
Matsuno \cite{Mat00} formulates his analytic Kida-type formula using N\'{e}ron period lattices. Since our sharp/flat $2$-adic $L$-functions are normalized by the Shimura periods fixed in Convention \ref{Convention:period_normalization}, we first relate these periods to N\'{e}ron periods and then study the variation of the latter under quadratic twists.

The first step relies on the following theorem.

\begin{thm}[Abbes--Ullmo {\cite[Th\'{e}or\`{e}me A]{AU96}}]\label{Thm:AU}
Let $A_0$ be an optimal elliptic curve over $\Q$ with good reduction at $2$. Then $2\nmid c(A_0)$.
\end{thm}
 
By Theorem \ref{Thm:AU} and Convention \ref{Convention:period_normalization}, the ratio of N\'{e}ron and Shimura periods is a $2$-adic unit for the optimal elliptic curve. Using Proposition \ref{Prop:mod_2_surjective}, we extend this result to non-optimal elliptic curves.

\begin{prop}\label{Prop:elliptic_modular_2-adic_unit}
For an elliptic curve $A$ over $\Q$ with good supersingular reduction at $2$, we see that $\Omega_A^{\pm}/\Omega_{f_A}^{\pm}$ is an element of $\Z_2$ and 
\begin{align}
\ord_2\left(\frac{\Omega_A^{\pm}}{\Omega_{f_A}^{\pm}}\right)=0.
\end{align}
\end{prop}

\begin{rem}
Greenberg--Vatsal \cite[Proposition 3.1 and Remark 3.4]{GV00}
established a similar comparison between canonical and
N\'{e}ron periods at odd primes, including non-optimal
elliptic curves when the residual representation is
irreducible. In the present setting, Proposition
\ref{Prop:elliptic_modular_2-adic_unit} establishes
the analogous $2$-adic unit property for the Shimura
periods normalized as in Convention
\ref{Convention:period_normalization}.
\end{rem}
\begin{rem}
Combined with the integrality results of \S \ref{sec:sharp/flat_L},
Proposition \ref{Prop:elliptic_modular_2-adic_unit}
also yields the integrality of sharp/flat $2$-adic $L$-functions normalized by N\'{e}ron periods for non-optimal
elliptic curves of square-free conductor in our
setting, strengthening \cite[Corollary 5.17]{Spr17}
in this case.
\end{rem}

\begin{proof}
By Theorem \ref{Thm:AU}, we see that
\begin{align}
\ord_2\left(\frac{\Omega_{A_0}^{\pm}}{\Omega_{f_{A_0}}^{\pm}}\right)=\ord_2(c(A_0))=0.
\end{align}
Let $\pi_A:X_0(N)\to A$ be a modular parametrization of minimal degree. Since $A_0$ is optimal, there exists a unique $\Q$-isogeny $\varphi_A:A_0\to A$ such that the following diagram
\[
\xymatrix{
X_0(N)\ar^{\ \ \pi_{A_0}}[r] \ar_{\pi_A}[rd]& A_0\ar[d]^{\varphi_A}\\
& A
}
\]
 is commutative. Then $\varphi_A$ also has minimal degree among all $\Q$-isogenies from $A_0$ to $A$. First, we show that $\deg\varphi_A$ is odd. By Proposition \ref{Prop:mod_2_surjective}, $A_0[2]$ has no non-trivial $G_{\Q}$-stable subgroup. Thus $A_0$ has no rational $2$-isogeny. Suppose that $2$ divides $\deg\varphi_A$. Since $(\ker\varphi_A)[2]=\ker\varphi_A\cap A_0[2]$, $(\ker\varphi_A)[2]$ is a nonzero $G_{\Q}$-stable subgroup of $A_0[2]$; hence we have $(\ker\varphi_A)[2]=A_0[2]$. Then there exists a $\Q$-isogeny $\Psi_A:A_0\to A$ such that $\varphi_A=\Psi_A\circ [2]$. This contradicts the minimality of $\deg\varphi_A$. Therefore we see that $\deg\varphi_A$ is odd. 
 
Since $A$ and $A_0$ have good reduction at $2$, we can extend $\varphi_A$ and $\widehat{\varphi_A}$ to morphisms between N\'{e}ron models over $\Z_2$. Then there exist constants $u_{\varphi_A},u_{\widehat{\varphi_A}}\in\Z_2$ such that 
\begin{align}
\varphi_A^*\omega_A=u_{\varphi_A}\omega_{A_0},\quad (\widehat{\varphi_A})^*\omega_{A_0}=u_{\widehat{\varphi_A}}\omega_A.
\end{align} 
Denoting by $\widehat{\varphi_A}$ the dual isogeny of $\varphi_A$, we see that 
 \begin{align}
 [\deg\varphi_A]^*\omega_{A_0}=(\widehat{\varphi_A}\circ \varphi_A)^*\omega_{A_0}=\varphi^*_A(\widehat{\varphi_A}^*\omega_{A_0})=\varphi_A^*(u_{\widehat{\varphi_A}}\omega_A)=u_{\widehat{\varphi_A}}u_{\varphi_A}\omega_{A_0};
 \end{align} 
hence we obtain $u_{\widehat{\varphi_A}}u_{\varphi_A}=\deg\varphi_A$. Combining these facts, it follows that $u_{\varphi_A}$ is a 2-adic unit. 

The push-forwards $(\varphi_A)_*:H_1(A_0(\C),\Z)^{\pm}\to H_1(A(\C),\Z)^{\pm}$ are injective and their image have finite index. We put
\begin{align}
m^{\pm}(\varphi_A):=[H_1(A(\C),\Z)^{\pm}:(\varphi_A)_*(H_1(A_0(\C),\Z)^{\pm})].
\end{align}
Then there exists $\varepsilon^{\pm}=\varepsilon_A^{\pm}\in\{\pm1\}$ such that
\begin{align}\label{eq:phi_A_delta}
(\varphi_A)_*(\delta_{A_0}^{\pm})=\varepsilon^{\pm}m^{\pm}(\varphi_A)\delta_A^{\pm}.
\end{align}
Similarly, putting 
\begin{align}
m^{\pm}(\widehat{\varphi_A}):=[H_1(A_0(\C),\Z)^{\pm}:(\widehat{\varphi_A})_*(H_1(A(\C),\Z)^{\pm})],
\end{align}
there exists $\eta^{\pm}\in\{\pm1\}$ such that
\begin{align}\label{eq:phi_A_delta_dual}
(\widehat{\varphi_A})_*(\delta_{A}^{\pm})=\eta^{\pm}m^{\pm}(\widehat{\varphi_A})\delta_{A_0}^{\pm}.
\end{align}
By \eqref{eq:phi_A_delta} and \eqref{eq:phi_A_delta_dual},
\begin{align}
(\deg\varphi_A)\delta_{A_0}^{\pm}=([\deg\varphi_A])_*(\delta_{A_0}^{\pm})=(\widehat{\varphi_A}\circ \varphi_A)_*(\delta_{A_0}^{\pm})=\varepsilon^{\pm}\eta^{\pm}m^{\pm}(\varphi_A)m^{\pm}(\widehat{\varphi_A})\delta_{A_0}^{\pm}.
\end{align}
Thus $m^{\pm}(\varphi_A)\mid \deg\varphi_A$; hence we see that $2\nmid m^{\pm}(\varphi_A)$.

Next, we show that $\kappa(A)=\kappa(A_0)$.
Let $A(\R)^0$, $A_0(\R)^0$ be the connected components of the identity for $A(\R)$, $A_0(\R)$, respectively. We show that $\varphi_A$ induces the isomorphism
\begin{align}
\overline{\varphi_A}:A_0(\R)/A_0(\R)^0\overset{\sim}{\longrightarrow}A(\R)/A(\R)^0.
\end{align}
Taking Galois cohomology of the following exact sequence of $\gal(\C/\R)$-modules
\[
\xymatrix{
0\ar[r] & X:=\ker\varphi_A\ar[r] &A_0(\C)\ar[r] & A(\C)\ar[r]&0,
}
\]
we have
\[
\xymatrix{
A_0(\R)\ar^{\varphi_A}[r] & A(\R)\ar[r] & H^1(\gal(\C/\R),X).
}
\]
Since $\sharp X=\deg\varphi_A$ and $\deg\varphi_A$ is odd, we see that $H^1(\gal(\C/\R),X)$ is trivial. Thus $\varphi_A$ is surjective. Since $\varphi_A$ is an isogeny, its restriction to $A_0(\R)^0$ is a nonconstant homomorphism of connected one-dimensional compact Lie groups, and hence is surjective. Therefore, $\varphi_A(A_0(\R)^0)=A(\R)^0$. For each $x+A_0(\R)^0\in \ker\overline{\varphi_A}$ with $x\in A_0(\R)$, $\varphi_A(x)\in A(\R)^0$. Then there exists $y\in A_0(\R)^0$ such that $\varphi_A(y)=\varphi_A(x)$. Then we obtain $x-y\in X(\R):=X^{\gal(\C/\R)}$. Since $\sharp X(\R)$ is odd and $\sharp A_0(\R)/A_0(\R)^0$ is $1$ or $2$, we see that $X(\R)\subset A_0(\R)^0$. This implies that $x-y\in A_0(\R)^0$. Thus we obtain
\begin{align}
x=y+(x-y)\in A_0(\R)^0.
\end{align}
This proves that $\overline{\varphi_A}$ is an isomorphism. Therefore we have
\begin{align}
u_{\varphi_A}\Omega_{A_0}^{\pm}&=\kappa(A_0)\int_{\delta_{A_0}^{\pm}}u_{\varphi_A}\omega_{A_0}\\
&=\kappa(A)\int_{\delta_{A_0}^{\pm}}\varphi_A^*\omega_A\\
&=\kappa(A)\int_{(\varphi_A)_*(\delta_{A_0}^{\pm})}\omega_A\\
&=\kappa(A)\int_{\varepsilon^{\pm}m^{\pm}(\varphi_A)\delta_A^{\pm}}\omega_A\\
&=\varepsilon^{\pm}m^{\pm}(\varphi_A)\kappa(A)\int_{\delta_A^{\pm}}\omega_A\\
&=\varepsilon^{\pm}m^{\pm}(\varphi_A)\Omega_A^{\pm}.
\end{align}
Then we obtain
\begin{align}\label{eq:elliptic_modular_ratio_differ}
\frac{\Omega_A^{\pm}}{\Omega_{f_A}^{\pm}}=\frac{\Omega_A^{\pm}}{c(A_0)^{-1}\Omega_{A_0}^{\pm}}=\frac{u_{\varphi_A}c(A_0)}{\varepsilon^{\pm}m^{\pm}(\varphi_A)},
\end{align}
and
\begin{align}
\ord_2\left(\frac{\Omega_A^{\pm}}{\Omega_{f_A}^{\pm}}\right)=0.
\end{align}
\end{proof}
Let $D$ be a positive square-free integer. Let $\psi_D:E^D\overset{\sim}{\to}E$ be an $\R$-isomorphism arising from the quadratic twist, and let $(\psi_D)_*:H_1(E^D(\C),\Z)\to H_1(E(\C),\Z)$ be the isomorphism induced by $\psi_D$.
\begin{lem}\label{Lem:Homology_action}
For each $s\in \{\pm\}$, we have
\begin{align}\label{eq:homology_twist}
(\psi_D)_*(H_1(E^D(\C),\Z)^s)=H_1(E(\C),\Z)^s.
\end{align}
In particular, there exists $\varepsilon_s\in\{\pm1\}$ such that
\begin{align}\label{eq:homology_generator_eigen}
(\psi_D)_*\delta_{E^D}^s=\varepsilon_s\delta_E^s.
\end{align}
\end{lem}
\begin{proof}
For an elliptic curve $A$ defined over $\Q$, let $c_A:A(\C)\to A(\C)$ be the complex conjugation and $(c_A)_*:H_1(A(\C),\Z)\to H_1(A(\C),\Z)$ be the morphism induced by $c_A$. Since $\psi_D\circ c_{E^D}=c_E\circ \psi_D$, on the homology group, we have
\begin{align}
(c_E)_*(\psi_D)_*=(\psi_D)_*(c_{E^D})_*.
\end{align}
For $s\in\{\pm\}$ and $\gamma\in H_1(E^D(\C),\Z)^s$, we have
\begin{align}
(c_E)_*(\psi_D)_*\gamma=(\psi_D)_*(c_{E^D})_*\gamma=(\psi_D)_*(s\gamma)=s(\psi_D)_*\gamma.
\end{align}
Then we obtain $(\psi_D)_*\gamma\in H_1(E(\C),\Z)^s$; hence $(\psi_D)_*(H_1(E^D(\C),\Z)^s)\subset H_1(E(\C),\Z)^s$. For $\eta\in H_1(E(\C),\Z)^s$, since $(\psi_D)_*$ is an isomorphism, there exists $\gamma\in H_1(E^D(\C),\Z)$ such that $\eta=(\psi_D)_*\gamma$. From the calculation
\begin{align}
(\psi_D)_*(c_{E^D})_*\gamma=(c_E)_*(\psi_D)_*\gamma=(c_E)_*\eta=s\eta=s(\psi_D)_*\gamma,
\end{align}
we see that 
\begin{align}
(\psi_D)_*((c_{E^D})_*\gamma-s\gamma)=0.
\end{align}
Then we have $(c_{E^D})_*\gamma=s\gamma$; hence we obtain $\gamma\in H_1(E^D(\C),\Z)^s$. Therefore
\begin{align}
\eta=(\psi_D)_*\gamma\in (\psi_D)_*(H_1(E^D(\C),\Z)^s);
\end{align}
hence we derive \eqref{eq:homology_twist}. The latter assertion follows from the fact that $H_1(A(\C),\Z)^{\pm}$ is a free $\Z$-module of rank one.
\end{proof}
Using Lemma \ref{Lem:Homology_action}, we study the ratio of N\'{e}ron periods.
\begin{prop}\label{Prop:periods_quadratic_twist}
Let $E$ be an elliptic curve over $\Q$ with good supersingular reduction at $2$. If $D$ is a positive square-free integer such that $D\equiv 1\bmod 4$, 
then there exists $\tilde{u}\in\Z_2^{\times}$ such that 
\begin{align}\label{eq:ratio_2-adic}
\frac{\Omega_{E^D}^{\pm}}{\Omega_E^{\pm}}=\frac{\tilde{u}\varepsilon_{\pm}}{\sqrt{D}}.
\end{align}
\end{prop}
\begin{proof}
Since $D\equiv1\bmod 4$, by \cite[Proposition 2.5 and Lemma 3.1]{Pal12}, there exists a constant $\tilde{u}\in\Z_2^{\times}$ such that
\begin{align}
\tilde{u}\psi_D^*(\omega_E)=\sqrt{D}\omega_{E^D}.
\end{align}
By the condition $D>0$, we have an $\R$-isomorphism $E\cong E^D$; hence we see that $\kappa(E)=\kappa(E^D)$. Therefore we obtain
\begin{align}
\Omega_{E^D}^{\pm}&=\kappa(E^D)\int_{\delta_{E^D}^{\pm}}\omega_{E^D}\\
&=\kappa(E)\int_{\delta_{E^D}^{\pm}}\frac{\tilde{u}}{\sqrt{D}}\psi_D^*(\omega_E)\\
&=\frac{\tilde{u}\kappa(E)}{\sqrt{D}}\int_{(\psi_D)_*\delta_{E^D}^{\pm}}\omega_E\\
&\overset{\eqref{eq:homology_generator_eigen}}{=}\frac{\tilde{u}\kappa(E)}{\sqrt{D}}\int_{\varepsilon_{\pm}\delta_E^{\pm}}\omega_E=\frac{\tilde{u}\varepsilon_{\pm}}{\sqrt{D}}\Omega_E^{\pm},
\end{align}
which completes the proof.
\end{proof}

\begin{exa}\label{example:period_ration1}
Let $E_1$ be the elliptic curve over $\Q$ with LMFDB label 11.a2 defined by
\begin{align}
y^2+y=x^3-x^2-10x-20.
\end{align}
Using SageMath \cite{Sage}, we numerically compute the absolute values
of the N\'eron period ratios with 100-bit precision and obtain
\begin{align}
\left|\frac{\Omega_{E_1^{13}}^{\pm}}{\Omega_{E_1}^{\pm}}\right|
&\approx 0.27735009811261456100917087
\approx \frac{1}{\sqrt{13}}, \notag\\
\left|\frac{\Omega_{E_1^{29}}^{\pm}}{\Omega_{E_1}^{\pm}}\right|
&\approx 0.18569533817705186314657622
\approx \frac{1}{\sqrt{29}}.
\end{align}
\end{exa}

\begin{exa}\label{example:period_ration2}
Let $E_2$ be the elliptic curve over $\Q$ with LMFDB label 19.a3 defined by
\begin{align}
y^2+y=x^3+x^2+x.
\end{align}
Similarly, numerical computations using SageMath \cite{Sage} give
\begin{align}
\left|\frac{\Omega_{E_2^{17}}^{\pm}}{\Omega_{E_2}^{\pm}}\right|
&\approx 0.24253562503633297351890646
\approx \frac{1}{\sqrt{17}}, \notag\\
\left|\frac{\Omega_{E_2^{97}}^{\pm}}{\Omega_{E_2}^{\pm}}\right|
&\approx 0.10153461651336190434789909
\approx \frac{1}{\sqrt{97}}.
\end{align}
\end{exa}

\begin{rem}
For each curve $A$ among $E_1$, $E_2$, and their quadratic twists
in Examples \ref{example:period_ration1} and
\ref{example:period_ration2}, we used its global minimal model
and computed a basis $(\omega_1(A),\omega_2(A))$ of its period
lattice with 100-bit precision using
\texttt{period\_lattice().basis(prec=100)}
in SageMath \cite{Sage}.
Since $\kappa(A)=1$ for all these curves, the absolute values
of the real and imaginary N\'eron periods are
$|\omega_1(A)|$ and $2|\operatorname{Im}(\omega_2(A))|$,
respectively, by Convention \ref{Convention:period_normalization}.
Thus, for each quadratic twist $E^D$ considered above,
we computed the absolute values of the period ratios as
\begin{align}
\left|\frac{\Omega_{E^D}^{+}}{\Omega_E^{+}}\right|
&=\frac{|\omega_1(E^D)|}{|\omega_1(E)|}, \notag\\
\left|\frac{\Omega_{E^D}^{-}}{\Omega_E^{-}}\right|
&=\frac{|\operatorname{Im}(\omega_2(E^D))|}
{|\operatorname{Im}(\omega_2(E))|}.
\end{align}
We then compared the computed ratios with $1/\sqrt{D}$
evaluated in \texttt{RealField(100)}.
The displayed numerical values have 26 significant digits.
For all four quadratic twists and both signs, the maximum
absolute difference between the computed ratios and
$1/\sqrt{D}$ was
$3.9443045261050590270586428264\times10^{-31}$.
These numerical computations are consistent with Proposition
\ref{Prop:periods_quadratic_twist}, but do not by themselves
establish exact identities for the period ratios.
\end{rem}

%%%%%%%%%%%%%%%%%%%%%%
\subsection{Mazur--Tate elements for quadratic twists of modular forms}\label{ss:MT_quadratic}
%%%%%%%%%%%%%%%%%%%%%%
We take $f=f_E$, where $E$ is an elliptic curve over $\Q$ with good supersingular reduction at $2$ and square-free conductor, and let $D>0$ be a
square-free integer such that $D\equiv1\bmod 4$. We use the periods normalized as in Subsection \ref{ss:Period_twists}. Let $\chi$ be the primitive quadratic character associated with $\Q(\sqrt{D})$. Using results of Atkin--Li \cite{AL78}, we prove that the coefficientwise twist $f_{\chi}=\sum_{n\geq1}a_n(f)\chi(n)q^n$ of $f$ by $\chi$ coincides with $f\otimes\chi:=f_{E^D}$.
\begin{prop}\label{Prop:primitive_twist_coefficient}
Let $E$ be an elliptic curve defined over $\Q$ with square-free conductor $N$ and let $D>0$ be a square-free integer with $D\equiv1\bmod 4$. Then we see that $f_{\chi}=f\otimes\chi$.
\end{prop}
\begin{proof}
Write $D=\ell_1\cdots \ell_r$ and 
\begin{align}
\chi=\prod_{j=1}^r\chi_{\ell_j},
\end{align}
where $\chi_{\ell_j}$ has conductor $\ell_j$. Set
\begin{align}
f_0:=f_E,\quad f_j:=(f_{j-1})_{\chi_{\ell_j}}\ (1\leq j\leq r). 
\end{align}
We show inductively that each $f_j$ is a newform with trivial nebentypus. Suppose that $f_{j-1}$ is a newform of level $N_{j-1}$ with trivial nebentypus. First suppose that $\ell_j\nmid N$. Since $(\ell_j,N_{j-1})=1$, a result of \cite[p.228]{AL78} implies that $f_j$ is a newform of level $N_{j-1}\ell_j^2$ with nebentypus $\chi_{\ell_j}^2=1$. In the case of $\ell_j\mid N$, using \cite[Corollary 4.1]{AL78}, we see that $f_j$ is a newform of level
\begin{align}
\max\{\ell_j,\ell_j^2\}\times\frac{N_{j-1}}{\ell_j}=\ell_jN_{j-1},
\end{align}
and that its nebentypus is $\chi_{\ell_j}^2=1$ by \cite[Theorem 4.1]{AL78}. Consequently, $f_{\chi}=f_r$ is a newform. Since $a_p(f_{\chi})=a_p(f\otimes\chi)$ for each prime $p\nmid ND$, the strong multiplicity one theorem (cf. \cite[Theorem 4.6.19]{Miy06}) implies that $f_{\chi}=f\otimes\chi$.
\end{proof}
By Proposition \ref{Prop:primitive_twist_coefficient}, it suffices to study the congruence relation between $f_{\chi}$ and $f$.
\begin{lem}
For $b\in (\Z/2^{n+2}\Z)^{\times}$, we have
\begin{align}\label{eq:modular_symbol_twist}
\left[\frac{b}{2^{n+2}}\right]_{f\otimes \chi}^{\sgn(\omega^i)}=\left(\frac{\Omega_f^{\sgn(\omega^i)}}{\Omega_{f\otimes\chi}^{\sgn(\omega^i)}}\right)\frac{1}{\tau(\chi)}\sum_{a\in \Z/D\Z}\chi(a)\left[\frac{bD+2^{n+2}a}{2^{n+2}D}\right]_f^{\sgn(\omega^i)}.
\end{align}
\end{lem}
\begin{proof}
Using Birch's lemma (cf. \cite[Chapter I, (8.3)]{MTT86})
\begin{align}
(f\otimes \chi)(z)=\frac{1}{\tau(\chi)}\sum_{a\in\Z/D\Z}\chi(a)f\left(z+\frac{a}{D}\right),
\end{align}
we obtain
\begin{align}
\left[\frac{b}{2^{n+2}}\right]_{f\otimes \chi}^{\pm}&=\frac{\eta(f\otimes\chi,b/2^{n+2})^{\pm}}{\Omega_{f\otimes\chi}^{\pm}}\\
&=\frac{\pi i}{\Omega_{f\otimes\chi}^{\pm}}\left(\int_{i\infty}^{\frac{b}{2^{n+2}}}(f\otimes\chi)(z)dz\pm\int_{i\infty}^{-\frac{b}{2^{n+2}}}(f\otimes\chi)(z)dz\right)\\
&=\frac{\pi i}{\tau(\chi)\Omega_{f\otimes\chi}^{\pm}}\sum_{a\in\Z/D\Z}\chi(a)\left(\int_{i\infty}^{\frac{b}{2^{n+2}}}f\left(z+\frac{a}{D}\right)dz\pm\int_{i\infty}^{-\frac{b}{2^{n+2}}}f\left(z+\frac{a}{D}\right)dz\right)\\
&\overset{(*)}{=}\frac{\Omega_f^{\pm}}{\tau(\chi)\Omega_{f\otimes\chi}^{\pm}}\sum_{a\in\Z/D\Z}\chi(a)\left[\frac{bD+2^{n+2}a}{2^{n+2}D}\right]^{\pm}_f.
\end{align}
To obtain the equality $(*)$, we make the change of variable $a\mapsto -a$ and use $\chi(-1)=1$ as
\begin{align}
\sum_{a\in\Z/D\Z}\chi(a)\int_{i\infty}^{-\frac{b}{2^{n+2}}}f\left(z+\frac{a}{D}\right)dz&=\sum_{a\in\Z/D\Z}\chi(-a)\int_{i\infty}^{-\frac{b}{2^{n+2}}}f\left(z+\frac{-a}{D}\right)dz\\
&=\sum_{a\in\Z/D\Z}\chi(a)\int_{i\infty}^{-\left(\frac{bD+2^{n+2}a}{2^{n+2}D}\right)}f(z)dz.
\end{align}
\end{proof}
By \eqref{eq:modular_symbol_twist}, Theorem \ref{Thm:periods}, and Propositions \ref{Prop:elliptic_modular_2-adic_unit} and \ref{Prop:periods_quadratic_twist}, we have
\begin{align}
2\left[\frac{b}{2^{n+2}}\right]_{f\otimes\chi}^{\pm}\in\mathbb Z_2
\end{align}
for every $b\in(\Z/2^{n+2}\Z)^{\times}$. Consequently, $\theta_{n,1}(f\otimes\chi,\omega^i,T)\in\Lambda_n$.
\begin{prop}\label{Prop:finite_cong}
For each $n\geq0$, we see that
\begin{align}
U_{n,\chi}:=U_{n,\chi}(T):=\left(\frac{\Omega_f^{\sgn(\omega^i)}}{\Omega_{f\otimes\chi}^{\sgn(\omega^i)}}\right)\frac{1}{\tau(\chi)}\omega^{-i}(D)(1+T)^{-t_n(D)}\in\Lambda_n^{\times}.
\end{align}
Furthermore, we have
\begin{align}\label{eq:MT_congruence}
\theta_{n,1}(f\otimes\chi,\omega^i,T)\equiv U_{n,\chi}(T)\theta_{n,D}(f,\omega^i,T)\bmod 2\Lambda_n.
\end{align}
\end{prop}
\begin{proof}
Since $\tau(\chi)^2=D$ and the sign ambiguity has no effect on the 2-adic valuations considered below, we henceforth write $\tau(\chi)=\sqrt{D}$. By \eqref{eq:elliptic_modular_ratio_differ} and Proposition \ref{Prop:periods_quadratic_twist}, we have
\begin{align}
U_{n,\chi}&=\left(\frac{\Omega_f^{\sgn(\omega^i)}}{\Omega_{f\otimes\chi}^{\sgn(\omega^i)}}\right)\frac{1}{\tau(\chi)}\omega^{-i}(D)(1+T)^{-t_n(D)}\\
&=\frac{m^{\sgn(\omega^i)}(\varphi_E)\varepsilon^{\pm}_E\cdot c((E^D)_0)u_{\varphi_{E^D}}}{m^{\sgn(\omega^i)}(\varphi_{E^D})\varepsilon^{\pm}_{E^D}\cdot c(E_0)u_{\varphi_E}}\frac{\Omega_E^{\sgn(\omega^i)}}{\Omega_{E^D}^{\sgn(\omega^i)}}\frac{1}{\sqrt{D}}\omega^{-i}(D)(1+T)^{-t_n(D)}\\
&=\frac{m^{\sgn(\omega^i)}(\varphi_E)\varepsilon^{\pm}_Ec((E^D)_0)u_{\varphi_{E^D}}}{m^{\sgn(\omega^i)}(\varphi_{E^D})\varepsilon^{\pm}_{E^D}c(E_0)u_{\varphi_E}}\frac{\sqrt{D}}{\tilde{u}\varepsilon_{\pm}}\frac{1}{\sqrt{D}}\omega^{-i}(D)(1+T)^{-t_n(D)}\in\Lambda_n^{\times}.
\end{align}
From the equality \eqref{eq:modular_symbol_twist}, we obtain
\begin{align}\label{eq:twisting_MT_relation}
\theta_{n,1}(f\otimes\chi&,\omega^i,T)\\
&=\sum_{b\in(\Z/2^{n+2}\Z)^{\times}}\left[\frac{b}{2^{n+2}}\right]_{f\otimes\chi}^{\sgn(\omega^i)}\omega^i(b)(1+T)^{t_n(b)}\\
&=\sum_{b\in(\Z/2^{n+2}\Z)^{\times}}\left(\frac{\Omega_f^{\sgn(\omega^i)}}{\Omega_{f\otimes\chi}^{\sgn(\omega^i)}}\right)\frac{1}{\tau(\chi)}\sum_{a\in \Z/D\Z}\chi(a)\left[\frac{bD+2^{n+2}a}{2^{n+2}D}\right]_f^{\sgn(\omega^i)}\omega^i(b)(1+T)^{t_n(b)}.
\end{align}
Under the change of variable $c=bD+2^{n+2}a$, we have
\begin{align}\label{eq:character_changing_variable}
\chi(a)=\chi(2^{-(n+2)}c)=\chi(2^{n+2})^{-1}\chi(c)\equiv \chi(c)\bmod 2,\quad \omega^i(b)=\omega^i(D)^{-1}\omega^i(c).
\end{align}
Moreover, since 
\begin{align}
\langle c\rangle=\langle b\rangle\langle D\rangle\text{ in }(1+4\Z_2)/(1+2^{n+2}\Z_2),
\end{align}
we have
\begin{align}\label{eq:log_changing_variable}
t_n(b)\equiv t_n(c)-t_n(D)\bmod 2^n.
\end{align}
Substituting \eqref{eq:character_changing_variable} and \eqref{eq:log_changing_variable} into \eqref{eq:twisting_MT_relation}, we have
\begin{align}
\theta_{n,1}(f\otimes\chi,\omega^i,T)&=U_{n,\chi}\sum_{c\in (\Z/2^{n+2}D\Z)^{\times}}\chi(2^{-(n+2)})\chi(c)\left[\frac{c}{2^{n+2}D}\right]_f^{\sgn(\omega^i)}\omega^i(c)(1+T)^{t_n(c)}\\
&=U_{n,\chi}\sum_{\substack{c\in (\Z/2^{n+2}D\Z)^{\times}\\ 1\leq c<2^{n+1}D}}2\chi(2^{-(n+2)})\chi(c)\left[\frac{c}{2^{n+2}D}\right]_f^{\sgn(\omega^i)}\omega^i(c)(1+T)^{t_n(c)}\\
&\equiv U_{n,\chi}\sum_{\substack{c\in (\Z/2^{n+2}D\Z)^{\times}\\ 1\leq c<2^{n+1}D}}2\left[\frac{c}{2^{n+2}D}\right]_f^{\sgn(\omega^i)}\omega^i(c)(1+T)^{t_n(c)}\\
&\equiv U_{n,\chi}(T)\theta_{n,D}(f,\omega^i,T)\bmod 2\Lambda_n.
\end{align}
\end{proof}

%%%%%%%%%%%%%%%%%%%%%%

%%%%%%%%%%%%%%%%%%%%%%%
\subsection{Matsuno-type formula for the sharp/flat 2-adic $L$-function}
\label{ss:Matsuno_formula}
%%%%%%%%%%%%%%%%%%%%%%

Let $E$ be an elliptic curve defined over $\Q$ with good supersingular reduction at 2 and square-free conductor, and let $D>0$ be a square-free integer such that $D\equiv 1\bmod 4$. Let $f_A$ be a weight-two newform corresponding to an elliptic curve $A$ defined over $\Q$ by the modularity theorem. For a nonzero $F\in\Z_2\llbracket T\rrbracket$, we denote its Iwasawa $\mu$- and $\lambda$-invariants by $\mu_2(F)$ and $\lambda_2(F)$, respectively, and set $\mu_2(0):=+\infty$. We put $\mu^{\sharp/\flat}_{2,M}(A,\omega^i):=\mu_2(L_{2,M}^{\sharp/\flat}(f_A,\omega^i,T))$ and $\lambda^{\sharp/\flat}_{2,M}(A,\omega^i):=\lambda_2(L_{2,M}^{\sharp/\flat}(f_A,\omega^i,T))$. For simplicity, denote $\mu_{2,1}^{\sharp/\flat}(A,\omega^i)$, $\lambda_{2,1}^{\sharp/\flat}(A,\omega^i)$ by $\mu_2^{\sharp/\flat}(A,\omega^i)$, $\lambda_2^{\sharp/\flat}(A,\omega^i)$, respectively.

A direct calculation gives the following lemma.
\begin{lem}
For each $n\geq0$, we have
\begin{align}\label{eq:matrix_integral}
R_n(\alpha,\beta)\begin{pmatrix}-1&-1\\ \beta&\alpha\end{pmatrix}^{-1}C(f)^{(n+3)}=\begin{pmatrix}-a_2(f)&-1\\ 1&0\end{pmatrix}.
\end{align}
\end{lem}
To study the variation of Iwasawa invariants under the congruence \eqref{eq:MT_congruence}, we show the following lemmas.
\begin{lem}
Put
\begin{align}
U_{\chi}:=\left(\frac{\Omega_f^{\sgn(\omega^i)}}{\Omega_{f\otimes\chi}^{\sgn(\omega^i)}}\right)\frac{1}{\tau(\chi)}\omega^{-i}(D)(1+T)^{-t(D)}.
\end{align}
Then we have $\mu_2(U_{\chi})=\lambda_2(U_{\chi})=0$.
\end{lem}
\begin{proof}
By a calculation similar to that in the proof of Proposition~\ref{Prop:finite_cong}, we have
\begin{align}
U_{\chi}=\frac{m^{\sgn(\omega^i)}(\varphi_E)\varepsilon^{\pm}_Ec((E^D)_0)u_{\varphi_{E^D}}}{m^{\sgn(\omega^i)}(\varphi_{E^D})\varepsilon^{\pm}_{E^D}c(E_0)u_{\varphi_E}}\frac{1}{\tilde{u}\varepsilon_{\pm}}\omega^{-i}(D)(1+T)^{-t(D)},
\end{align}
and this implies the lemma.
\end{proof}
Let $\gamma_n$ denote the image of the topological generator $5$ in $\Gamma_n$.
\begin{lem}
For each $n\in\N$ and $F=\sum_{s=0}^{2^{n-1}-1}a_s\gamma_{n-1}^s\in \Lambda_{n-1}$, we have
\begin{align}\label{eq:U_twist}
\nu_{n-1/n}(U_{n-1,\chi}F)=U_{n,\chi}\nu_{n-1/n}(F).
\end{align}
\end{lem}
\begin{proof}
We put
\begin{align}
c_{\chi}:=\left(\frac{\Omega_f^{\sgn(\omega^i)}}{\Omega_{f\otimes\chi}^{\sgn(\omega^i)}}\right)\frac{1}{\tau(\chi)}\omega^{-i}(D).
\end{align}
Therefore we see that
\begin{align}
\nu_{n-1/n}(U_{n-1,\chi}F)&=c_{\chi}\sum_{s=0}^{2^{n-1}-1}a_s\nu_{n-1/n}(\gamma_{n-1}^{s-t_{n-1}(D)})\\
&=c_{\chi}\sum_{s=0}^{2^{n-1}-1}a_s(\gamma_n^{s-t_n(D)}+\gamma_n^{s-t_n(D)+2^{n-1}})\\
&=c_{\chi}\gamma_n^{-t_n(D)}\sum_{s=0}^{2^{n-1}-1}a_s(\gamma_n^s+\gamma_n^{s+2^{n-1}})\\
&=U_{n,\chi}\nu_{n-1/n}(F).
\end{align}
\end{proof}
\begin{lem}
Fix $*\in\{\sharp,\flat\}$ and assume that $L_2^*(E,\omega^i,T)\neq0$. If $D=\ell_1\cdots\ell_r$ is the prime factorization of $D$, then
\begin{align}
\mu_{2,D}^*(E,\omega^i)
&=
\mu_2^*(E,\omega^i)
+
\sum_{j=1}^r \mu_2\!\left(h_{\ell_j}(E,\omega^i,T)\right),
\label{eq:mu_conductor_differ}
\\
\lambda_{2,D}^*(E,\omega^i)
&=
\lambda_2^*(E,\omega^i)
+
\sum_{j=1}^r \lambda_2\!\left(h_{\ell_j}(E,\omega^i,T)\right).
\label{eq:lambda_conductor_differ}
\end{align}
\end{lem}
\begin{proof}
We use the equality \eqref{eq:prime_factor_reduce} repeatedly.
\end{proof}
We derive a Greenberg--Vatsal-type congruence (cf. \cite[Theorem (3.10)]{GV00}) for the sharp/flat $p$-adic $L$-functions from finite-layer congruences. Related congruence results for sharp/flat $p$-adic $L$-functions at non-ordinary primes were obtained by Corpuz--Lei \cite[Theorem 4.5]{CL25+} under the Fontaine--Laffaille condition $p>k$, where $k$ denotes the weight of the newform. Thus, their results do not cover the case $p=k=2$.

\begin{thm}\label{Theorem:congruence_relation}
For $*\in\{\sharp,\flat\}$, we have
\begin{align}\label{eq:cong_sharp}
L^*_{2,1}(f\otimes\chi,\omega^i,T)\equiv U_{\chi}L^*_{2,D}(f,\omega^i,T)\bmod 2.
\end{align}
In particular, for $*\in\{\sharp,\flat\}$,
\begin{align}\label{cong_twist_L}
\overline{L^*_2(E^D,\omega^i,T)}\equiv u_{D,i}(T)T^{\Delta_E(D)}\overline{L^*_2(E,\omega^i,T)}
\end{align}
in $\F_2\llbracket T\rrbracket$. Here,
\begin{align}
\overline{h_{\ell}(f_E,\omega^i,T)}=T^{\lambda_2(h_{\ell}(f_E,\omega^i,T))}u_{\ell,i}(T),\quad u_{\ell,i}(T)\in\F_2\llbracket T\rrbracket^{\times},
\end{align}
and
\begin{align}
u_{D,i}(T):=\overline{U_{\chi}}\prod_{\ell\mid D}u_{\ell,i}(T),\quad \Delta_E(D):=\sum_{\ell\mid D}\lambda_2(h_{\ell}(f_E,\omega^i,T)).
\end{align}
\end{thm}
\begin{proof}
We fix $(\vec{L}_{2,n,1}^{\omega^i}(f\otimes\chi))_n=((L^{\sharp}_{2,n,1}(f\otimes\chi,\omega^i,T),L^{\flat}_{2,n,1}(f\otimes\chi,\omega^i,T)))_n\in(\Lambda_n^{\oplus 2})_n$ and $(\vec{L}^{\omega^i}_{2,n,D}(f))_n=((L^{\sharp}_{2,n,D}(f,\omega^i,T),L^{\flat}_{2,n,D}(f,\omega^i,T)))_n\in(\Lambda_n^{\oplus 2})_n$ obtained by Proposition \ref{Prop:finite_exist}.
Combining \eqref{eq:exist_finite_sharp/flat}, \eqref{eq:matrix_integral}, and \eqref{eq:U_twist}, we see that
\begin{align}
&\vec{L}_{2,n,1}^{\omega^i}(f\otimes\chi)C_1(f\otimes\chi)\cdots C_n(f\otimes\chi)\\
&=(\theta_{n,1}(f\otimes \chi,\omega^i,T),\nu_{n-1/n}(\theta_{n-1,1}(f\otimes\chi,\omega^i,T)))\begin{pmatrix}-a_2(f\otimes\chi)&-1\\ 1&0\end{pmatrix}\\
&\equiv U_{n,\chi}(\theta_{n,D}(f,\omega^i,T),\nu_{n-1/n}(\theta_{n-1,D}(f,\omega^i,T)))\begin{pmatrix}-a_2(f)&-1\\ 1&0\end{pmatrix}\bmod 2\\
&\equiv U_{n,\chi}\vec{L}_{2,n,D}^{\omega^i}(f)C_1(f)\cdots C_n(f)\bmod 2.
\end{align}
Since $C_i(f\otimes\chi)\equiv C_i(f)\bmod 2$ for each $i$, we have
\begin{align}\label{eq:mod_2_vanish}
(\vec{L}^{\omega^i}_{2,n,1}(f\otimes\chi)-U_{n,\chi}\vec{L}_{2,n,D}^{\omega^i}(f))C_1(f)\cdots C_n(f)\equiv 0\bmod 2,
\end{align}
for all $n\geq1$. Since 
\begin{align}\label{eq:mod-2_congruence_n-th-layer}
C_j(f)\equiv \begin{pmatrix}0&1\\ T^{2^{j-1}}&0\end{pmatrix}\bmod 2,
\end{align}
we obtain
\begin{align}\label{eq:prod_C_i}
C_1(f)\cdots C_n(f)\equiv\begin{cases}
\begin{pmatrix}T^{q_n^{\sharp}}&0\\ 0& T^{q_n^{\flat}}\end{pmatrix}\bmod 2&(2\mid n),\\
\begin{pmatrix}0&T^{q_n^{\sharp}}\\ T^{q_n^{\flat}}&0\end{pmatrix}\bmod 2&(2\nmid n),
\end{cases}
\end{align}
where
\begin{align}
q_n^{\sharp}:=\begin{cases}2+2^3+\cdots +2^{n-1}&(2\mid n),\\ 2+2^3+\cdots+2^{n-2}&(2\nmid n)\end{cases},\quad q_n^{\flat}:=\begin{cases}1+2^2+\cdots +2^{n-2}&(2\mid n),\\ 1+2^2+\cdots+2^{n-1}&(2\nmid n).\end{cases}
\end{align}
By the following calculation
\begin{align}
2^n-q_n^{\sharp}=\begin{cases}\frac{2^n+2}{3}&(2\mid n),\\ \frac{2^{n+1}+2}{3}&(2\nmid n),\end{cases}\quad 2^n-q_n^{\flat}=\begin{cases}\frac{2^{n+1}+1}{3}&(2\mid n),\\ \frac{2^n+1}{3}&(2\nmid n),\end{cases}
\end{align}
we obtain $2^n-q_n^{\sharp/\flat}\to\infty$ as $n\to\infty$.

Put
\begin{align}
h_n:=2^n-\max\{q_n^{\sharp},q_n^{\flat}\}=\left\lceil\frac{2^n}{3}\right\rceil.
\end{align}
For $g\in\{f,f\otimes\chi\}$, write $a_g:=a_2(g)$, $P_n(g):=C_1(g)\cdots C_n(g)$, and put
\begin{align}
K_g:=\begin{pmatrix}-a_g&-1\\1&0\end{pmatrix},\quad
W_{n,M}(g):=(\theta_{n,M}(g,\omega^i,T),\nu_{n-1/n}(\theta_{n-1,M}(g,\omega^i,T))).
\end{align}
We consider the set of finite-layer solutions
\begin{align}
S_{n,M}(g):=\{X\in\Lambda_n^{\oplus2}\mid XP_n(g)=W_{n,M}(g)K_g\}.
\end{align}
By \eqref{eq:exist_finite_sharp/flat} and \eqref{eq:matrix_integral}, this set is nonempty and contains the vectors given by Proposition \ref{Prop:finite_exist}. If $X,Y\in S_{n,M}(g)$, then $(X-Y)P_n(g)=0$. Since $a_g\equiv0\bmod2$, the computation in \eqref{eq:prod_C_i} applies to $P_n(g)$. Using
$\Lambda_n/2\Lambda_n\cong\F_2[T]/(T^{2^n})$, we obtain
\begin{align*}
X-Y\in(2,T^{h_n})\Lambda_n^{\oplus2}.
\end{align*}

Let $\prj_n:\Lambda\to\Lambda_n$ be the natural projection and put
\begin{align}
C(g):=\begin{pmatrix}a_g&1\\-2&0\end{pmatrix},\qquad
Q_g:=\begin{pmatrix}a_g&2\\-1&0\end{pmatrix}.
\end{align}
The queue relation and $\pi_{n+1/n}\nu_{n/n+1}=2$ give
\begin{align}
\pi_{n+1/n}(W_{n+1,M}(g))=W_{n,M}(g)Q_g.
\end{align}
Moreover, $\pi_{n+1/n}(P_{n+1}(g))=P_n(g)C(g)$ and $Q_gK_g=K_gC(g)$.
Since $\det C(g)=2$ and $\Lambda_n$ is $2$-torsion-free, right multiplication by $C(g)$ is injective on $\Lambda_n^{\oplus2}$. Thus
\begin{align}
\pi_{n+1/n}(S_{n+1,M}(g))\subseteq S_{n,M}(g).
\end{align}
Each $S_{n,M}(g)$ is a nonempty closed subset of the compact $\Z_2$-module $\Lambda_n^{\oplus2}$. Since the projections preserve these solution sets, the finite intersection property gives a compatible family of solutions. It is unique: the difference of two such families belongs, by the kernel estimate above and $h_n\to\infty$, to $2\Lambda^{\oplus2}$; dividing by $2$ and repeating (using the $2$-torsion-freeness of $\Lambda_n$) shows that the difference is zero. By Proposition \ref{Prop:finite_exist}, this unique compatible family is $\vec{L}_{2,M}^{\omega^i}(g)$ constructed in Subsection \ref{ss:2-adic_L-function}. Consequently, for either pair $(g,M)=(f\otimes\chi,1)$ or $(g,M)=(f,D)$, the finite-layer vector chosen above satisfies
\begin{align*}
\vec{L}^{\omega^i}_{2,n,M}(g)\equiv\prj_n\bigl(\vec{L}^{\omega^i}_{2,M}(g)\bigr)
\pmod{(2,T^{h_n})\Lambda_n^{\oplus2}}.
\end{align*}

Put
\begin{align}
\vec F_n:=\vec{L}_{2,n,1}^{\omega^i}(f\otimes\chi)-U_{n,\chi}\vec{L}_{2,n,D}^{\omega^i}(f),\quad
\vec F:=\vec{L}_{2,1}^{\omega^i}(f\otimes\chi)-U_{\chi}\vec{L}_{2,D}^{\omega^i}(f).
\end{align}
By \eqref{eq:mod_2_vanish}, we have $\vec F_nP_n(f)\equiv0\bmod2$. Applying the same calculation as for the kernel estimate above, we obtain
\begin{align}
\vec F_n\in(2,T^{h_n})\Lambda_n^{\oplus2}.
\end{align}
Since $\prj_n(U_{\chi})=U_{n,\chi}$, the finite-to-limit comparison above implies
\begin{align}
\prj_n(\vec F)\in(2,T^{h_n})\Lambda_n^{\oplus2}.
\end{align}
For any fixed positive integer $d$, taking $n$ sufficiently large that $h_n\geq d$ and $2^n\geq d$, we use the natural isomorphism
\begin{align}\label{identification_finite_infinity_reduction}
\Lambda/(2,T^d)\cong\Lambda_n/(2,T^d)
\end{align}
to conclude that $\vec F\equiv0\bmod(2,T^d)$. As this holds for every $d$, we get
\begin{align}
\vec F\in\bigcap_{d\geq1}(2,T^d)\Lambda^{\oplus2}=2\Lambda^{\oplus2},
\end{align}
which proves \eqref{eq:cong_sharp}.
The congruence \eqref{cong_twist_L} follows from Proposition \ref{Euler_factor_difference} and the fact that $\mu_2(h_{\ell}(f_E,\omega^i,T))=0$ for each $\ell\mid D$.
\end{proof}

\begin{prop}
Fix $*\in\{\sharp,\flat\}$ and a tame character $\omega^i:\Delta\to \C_2^{\times}$. Assume that $\mu_2^*(E,\omega^i)=0$. Then we have
\begin{align}\label{eq:lambda_twist_chai}
\mu_2^*(E^D,\omega^i)=\mu_{2,D}^*(E,\omega^i)=0,\quad \lambda^*_2(E^D,\omega^i)=\lambda_{2,D}^*(E,\omega^i).
\end{align}
\end{prop}
\begin{proof}
Combining \eqref{eq:mu_conductor_differ}
\begin{align}
\mu_{2,D}^*(E,\omega^i)=\mu_2^*(E,\omega^i)+\sum_{\ell\mid D}\mu_2(h_{\ell}(E,\omega^i,T))
\end{align}
with $\mu_2(h_{\ell}(E,\omega^i,T))=0$ for each $\ell\mid D$, we have $\mu_{2,D}^*(E,\omega^i)=0$. Moreover, by the congruence \eqref{eq:cong_sharp}, we also obtain $\mu_2^*(E^D,\omega^i)=0$. Since $\lambda_2(U_{\chi})=0$, the congruence \eqref{eq:cong_sharp} implies 
\begin{align}
\lambda^*_2(E^D,\omega^i)=\lambda_{2,D}^*(E,\omega^i).
\end{align}

\end{proof}
\begin{lem}\label{Lem:Iwasawa_Euler_factor}
Assume that $N_E$ is square-free. We put, for each rational prime $\ell\neq2$,
\begin{align}
n_{\ell}:=\ord_2\left(\frac{\ell^2-1}{8}\right).
\end{align}
For each prime $\ell$ of good reduction for $E$, let $\tilde{E}_{\ell}$ be the reduction of $E$ modulo $\ell$. Then we have
\begin{align}\label{eq:lambda_differ}
\lambda_2(h_{\ell}(E,\omega^i,T)):=\lambda_2(h_{\ell}(f_E,\omega^i,T))=\begin{cases}
0&(\ell\nmid N_E,\ 2\nmid\sharp \tilde{E}_{\ell}(\F_{\ell})),\\
2^{n_{\ell}+1}&(\ell\nmid N_E,\ 2\mid\sharp \tilde{E}_{\ell}(\F_{\ell})),\\
2^{n_{\ell}}&(\ell \mid N_E).
\end{cases}
\end{align}
\end{lem}
\begin{rem}
In the case of odd primes and higher weight modular forms, a general calculation was carried out by Pratap--Ray \cite[Propositions 3.3 and 3.4]{PR26} (see also \cite[Lemmas 3.3 and 3.4]{Mat00}). Pratap and Ray gave a detailed proof in the higher weight case, whereas we give an explicit formula and a simpler proof in the weight-two case at $p=2$.

\end{rem}
\begin{proof}
For simplicity, we put
\begin{align}
X:=1+T,\quad t:=t(\ell),\quad \omega:=\omega^i(\ell)\in\{\pm1\}.
\end{align}
Note that
\begin{align}
h_{\ell}(T):=h_{\ell}(E,\omega^i,T)=\begin{cases}
a_{\ell}-\omega X^t-\omega^{-1}X^{-t}&(\ell\nmid N_E),\\
a_{\ell}-\omega X^t&(\ell\mid N_E).
\end{cases}
\end{align}
Since 
\begin{align}
\ord_2\left(\frac{\ell^2-1}{8}\right)=\begin{cases}
\ord_2(\ell-1)-2&(\text{if }\ell\equiv 1\bmod 4),\\
\ord_2(\ell+1)-2&(\text{if }\ell\equiv 3\bmod 4),
\end{cases}
\end{align}
and 
\begin{align}
\ord_2(t)=\ord_2(\log_2(5^t))-2=\ord_2(5^t-1)-2,
\end{align}
we have
\begin{align}
\lambda_2(X^t-1)=2^{\ord_2(t)}=2^{\ord_2\left(\frac{\ell^2-1}{8}\right)}.
\end{align}

\noindent (i) If $\ell\nmid 2N_E$, since $\varepsilon(\ell)=1$ and $X^t\in \Z_2\llbracket T\rrbracket^{\times}$, we have $\lambda_2(X^th_{\ell}(T))=\lambda_2(h_{\ell}(T))$. From the calculation
\begin{align}
X^th_{\ell}(T)&= a_{\ell}X^t-\omega X^{2t}-\omega^{-1}\\
&\equiv a_{\ell}X^t+X^{2t}+1\bmod 2\\
&\equiv \sharp \tilde{E}_{\ell}(\F_{\ell})X^t+X^{2t}+1\bmod 2\\
&\equiv \begin{cases}
X^{2t}+X^t+1\bmod 2&(\text{if } 2\nmid \sharp \tilde{E}_{\ell}(\F_{\ell})),\\
(X^t-1)^2\bmod 2&(\text{if } 2\mid \sharp \tilde{E}_{\ell}(\F_{\ell})),
\end{cases}
\end{align}
we obtain
\begin{align}
\lambda_2(h_{\ell}(T))=\begin{cases}
0&(\text{if } 2\nmid \sharp \tilde{E}_{\ell}(\F_{\ell})),\\
2^{n_{\ell}+1}&(\text{if } 2\mid \sharp \tilde{E}_{\ell}(\F_{\ell})).
\end{cases}
\end{align}

\noindent (ii) If $\ell\mid N_E$, since $\varepsilon(\ell)=0$ and $a_{\ell}\equiv \pm1\bmod 2$, we have
\begin{align}
\lambda_2(h_{\ell}(T))&=\lambda_2(a_{\ell}-\omega X^t)\\
&=\lambda_2(X^t-1)\\
&=2^{n_{\ell}}.
\end{align}
\end{proof}
We prove the Matsuno-type formula for the sharp/flat $2$-adic $L$-functions.
\begin{thm}\label{Thm:Matsuno_formula}
Let $E$ be an elliptic curve over $\Q$ with good supersingular reduction at $2$, and let $D>0$ be a square-free integer such that $D\equiv 1\bmod 4$. Fix $*\in\{\sharp,\flat\}$ and a tame character $\omega^i:\Delta\to \C_2^{\times}$. Assume that the conductor $N_E$ of $E$ is square-free and $\mu_2^*(E,\omega^i)=0$. Then we have $\mu_2^*(E^D,\omega^i)=0$, and
\begin{align}\label{eq:Matsuno-type_formula}
\lambda_2^*(E^D,\omega^i)=\lambda_2^*(E,\omega^i)+\sum_{\substack{\ell\mid D\\ \ell\mid N_E}}2^{n_{\ell}}+\sum_{\substack{\ell\mid D\\ \ell\nmid N_E,2\mid \sharp\tilde{E}_{\ell}(\F_{\ell})}} 2^{n_{\ell}+1}.
\end{align}
\end{thm}
\begin{rem}
The quantity $2^{n_{\ell}}$ is the number of primes of $\Q_{\infty}$ lying above $\ell$. 
Indeed, if $\Frob_{\ell}$ denotes the Frobenius element at $\ell$ in $\Gamma$,
\begin{align}
[\Gamma:\overline{\langle \Frob_{\ell}\rangle}]=[\Z_2:t(\ell)\Z_2]=2^{n_{\ell}}.
\end{align}
Thus, Theorem \ref{Thm:Matsuno_formula} may be viewed as a Kida-type (Riemann--Hurwitz type) formula.
\end{rem}
\begin{proof}
The theorem follows by combining the formulas obtained in this subsection as
\begin{align}
\lambda^*_2(E^D,\omega^i)&\overset{\eqref{eq:lambda_twist_chai}}{=}\lambda_{2,D}^*(E,\omega^i)\\
&\overset{\eqref{eq:lambda_conductor_differ}}{=}\lambda_2^*(E,\omega^i)+\sum_{\ell\mid D}\lambda_2(h_{\ell}(E,\omega^i,T))\\
&\overset{\eqref{eq:lambda_differ}}{=}\lambda_2^*(E,\omega^i)+\sum_{\substack{\ell\mid D\\ \ell\mid N_E}}2^{n_{\ell}}+\sum_{\substack{\ell\mid D\\ \ell\nmid N_E,2\mid \sharp\tilde{E}_{\ell}(\F_{\ell})}} 2^{n_{\ell}+1}.
\end{align}
\end{proof}
\begin{rem}
Under the hypotheses
\begin{align}
\mu_2^{\sharp}(E,\omega^0)=\mu_2^{\flat}(E,\omega^0)=\mu_2^{\sharp}(E,\omega^1)=\mu_2^{\flat}(E,\omega^1)=0,
\end{align}
Theorem \ref{Thm:Matsuno_formula} also implies that all four sharp/flat $\lambda$-invariants vary in parallel under quadratic twists:
\begin{align}
&(\lambda_2^{\sharp}(E^D,\omega^0),\lambda_2^{\flat}(E^D,\omega^0),\lambda_2^{\sharp}(E^D,\omega^1),\lambda_2^{\flat}(E^D,\omega^1))\\
&\ \ \ \ -(\lambda_2^{\sharp}(E,\omega^0),\lambda_2^{\flat}(E,\omega^0),\lambda_2^{\sharp}(E,\omega^1),\lambda_2^{\flat}(E,\omega^1))=\Delta_E(D)(1,1,1,1).
\end{align}
\end{rem}
As an application of Theorem \ref{Thm:Matsuno_formula}, we can make arbitrarily large sharp/flat Iwasawa $\lambda$-invariants by quadratic twists as follows.
 \begin{lem}\label{Lem:linear_disjoint}
For an elliptic curve $E$ over $\Q$ with good supersingular reduction at $2$ and  a nonnegative integer $m$, we see that
\begin{align}
\Q(E[2])\cap \Q(\zeta_{2^m})=\Q.
\end{align}
\end{lem}
\begin{proof}
By the result of Serre \cite[Proposition 12]{Ser72}, we have $\bar{\rho}_{E,2}(I_2)\cong \F_4^{\times}\cong A_3$, where $I_2$ is the inertia group at $2$ and $A_3$ is the alternating group of order $3$. Let $\mathfrak{S}_3$ be the symmetric group on three letters. Since the image of $I_2\to \mathfrak{S}_3\twoheadrightarrow \mathfrak{S}_3/A_3$ is trivial, $2$ is unramified in the unique quadratic subfield of $\Q(E[2])$, if it exists. Put
\begin{align}
L:=\Q(E[2])\cap \Q(\zeta_{2^m}).
\end{align}
Since $[\Q(E[2]):\Q]\mid 6$ and $[\Q(\zeta_{2^m}):\Q]$ is a power of $2$, we have $[L:\Q]\leq 2$.
For $m\geq 3$, the quadratic subfields of $\Q(\zeta_{2^m})/\Q$ are 
\begin{align}
\Q(i),\quad \Q(\sqrt{2}),\quad \Q(\sqrt{-2}),
\end{align}
and all of them are ramified at $2$. The cases $m\leq 2$ are trivial. Therefore, we see that $\Q(E[2])\cap \Q(\zeta_{2^m})=\Q$.
\end{proof}

\begin{cor}\label{Cor:lambda_arbitrary_large}
Let $E$ be an elliptic curve defined over $\Q$ with good supersingular reduction at $2$ and square-free conductor. Fix $*\in\{\sharp,\flat\}$ and a character $\omega^i:(\Z/4\Z)^{\times}\to\C_2^{\times}$. Assume that $\mu_2^*(E,\omega^i)=0$. Then the set
\begin{align}
\{\lambda_2^*(E^{\ell},\omega^i)\ |\ \ell\text{ is a rational prime and }\ell\equiv1\bmod 4\}
\end{align}
is infinite. More precisely, for each nonnegative integer $k$, the density of the set
\begin{align}
\{\ell\text{ is a rational prime and }\ell\equiv 1\bmod 4\ |\ \lambda_2^*(E^{\ell},\omega^i)=\lambda_2^*(E,\omega^i)+2^{k+1}\}
\end{align}
is $\frac{1}{3\cdot 2^{k+1}}$.
\end{cor}
\begin{proof}
By Proposition \ref{Prop:mod_2_surjective} and Lemma \ref{Lem:linear_disjoint}, we see that
\begin{align}
\gal(\Q(E[2])\Q(\zeta_{2^{k+3}})/\Q)\cong \mathfrak{S}_3\times (\Z/2^{k+3}\Z)^{\times}.
\end{align}
For each rational prime $\ell\nmid 2N_E$, the condition $\widetilde{E}_{\ell}(\F_{\ell})[2]\neq\{O\}$ holds if and only if the Frobenius element $\Frob_{\ell}|_{\mathfrak{S}_3}$ fixes at least one element of $E[2]\backslash\{O\}$. Therefore, by the Chebotarev density theorem, the  density of the set
\begin{align}
\mathcal{W}:=\{\ell\text{ is a rational prime}\ |\ \ell\nmid N_E,\ \ell\equiv 1+2^{k+2}\bmod 2^{k+3},\ \widetilde{E}_{\ell}(\F_{\ell})[2]\neq\{O\}\}
\end{align}
is 
\begin{align}
\frac{4}{6}\cdot \frac{1}{2^{k+2}}=\frac{1}{3\cdot 2^{k+1}}.
\end{align}
For each element $\ell\in\mathcal{W}$, we see that 
\begin{align}
n_{\ell}=\ord_2(\ell-1)+\ord_2(\ell+1)-3=(k+2)+1-3=k.
\end{align}
For each rational prime $\ell\nmid 2N_E$ with
$\ell\equiv1\bmod 4$, Theorem \ref{Thm:Matsuno_formula}
shows that $\ell\in\mathcal{W}$ if and only if
\begin{align}
\lambda_2^*(E^{\ell},\omega^i)
-\lambda_2^*(E,\omega^i)=2^{k+1}.
\end{align}
\end{proof}
%%%%%%%%%%%%%%%%%%%%%%

%%%%%%%%%%%%%%%%%%%%%%
\section{Distributions of sharp/flat Iwasawa invariants}\label{sec:distribution}
%%%%%%%%%%%%%%%%%%%%%%
In this section, as an application of the Matsuno-type formula \eqref{eq:Matsuno-type_formula}, we study the distribution of the Iwasawa $\lambda$-invariant of the sharp/flat $2$-adic $L$-functions for elliptic curves. This is a supersingular, 2-adic analytic analogue of \cite[Theorem 4.1]{HR24+}. We use some results of Hatley--Ray in \cite[\S\S2-4]{HR24+}.

For functions $F(X)$ and $G(X)$ defined for sufficiently large real numbers $X$, we write $F(X)\sim G(X)$ to signify that 
\begin{align}
\lim_{X\to\infty}\frac{F(X)}{G(X)}=1,
\end{align}
and write $F(X)\gg G(X)$ if there exist constants $c>0$ and $X_0\in\R$ satisfying $F(X)\geq cG(X)$ for all $X\geq X_0$.
For a set $\mho$ of rational primes, we define
\begin{align}
n_{\mho}(X):=\sharp\{d:=\ell_1\cdots\ell_r\ |\ r\in\Z_{\geq0},\ \ell_i\notin \mho,\ 1\leq d\leq X,\ d\text{ is square-free}\},
\end{align}
and the natural density $\alpha(\mho)$ of $\mho$ by
\begin{align}
\alpha(\mho):=\lim_{X\to\infty}\frac{\sharp\{\ell\in\mho\ |\ \ell\leq X\}}{\pi(X)},
\end{align}
where $\pi$ is the prime-counting function. Throughout this section, an empty product is understood to be $1$. For an elliptic curve $E$ defined over $\Q$ with conductor $N_E$, we define the set of rational primes
\begin{align}
\mho_E:=\{\ell \nmid 2N_E\ |\ \widetilde{E}_{\ell}(\F_{\ell})[2]\neq \{O\}\}.
\end{align}
By Proposition \ref{Prop:mod_2_surjective} and the Chebotarev density theorem, we have the following lemma.
\begin{lem}\label{Lem:density_mho}
Let $E$ be an elliptic curve with good supersingular reduction at $2$. Then we have $\alpha(\mho_E)=2/3$.
\end{lem}
Using the Selberg--Delange-type theorem, we have the following.
\begin{prop}[cf. {\cite[Proposition 2.12, Corollary 3.8]{HR24+}}]\label{Lem:density_omega_E}
There exists a constant $c>0$ such that 
\begin{align}
n_{\mho_E}(X)\sim \frac{cX}{(\log X)^{\frac{2}{3}}}.
\end{align}
\end{prop}
\begin{proof}
Since $\mho_E$ is a Frobenian set of primes, by \cite[Th\'{e}or\`{e}me 2.8]{Ser75}, there exists a constant $c>0$ such that
\begin{align}
n_{\mho_E}(X)\sim \frac{cX}{(\log X)^{\alpha(\mho_E)}}.
\end{align}
Using the fact that $\alpha(\mho_E)=2/3$ by Lemma \ref{Lem:density_mho} completes the proof.
\end{proof}
\begin{thm}\label{Thm:distribution_lambda}
Let $E$ be an elliptic curve defined over $\Q$ with good supersingular reduction at $2$ and square-free conductor $N_E$. Fix $*\in\{\sharp,\flat\}$ and a tame character $\omega^i:\Delta\to\{\pm1\}$. Let $N^*$ be a nonnegative integer such that $N^*\geq \lambda_2^*(E,\omega^i)$ and $\lambda_2^*(E,\omega^i)\equiv N^*\bmod 2$. Assume that $\mu^*_2(E,\omega^i)=0$. Let
\begin{align}
m_{E,N^*}^*(X):=\sharp\{D\in\Z_{>0}\ |\ D\leq X,\ D\text{ is square-free, }D\equiv1\bmod 4,\ \lambda_2^*(E^D,\omega^i)=N^*\}.
\end{align}
Then we have
\begin{align}
m_{E,N^*}^*(X)\gg \frac{X}{(\log X)^{\frac{2}{3}}}.
\end{align}
\end{thm}
\begin{proof}
The proof is a modification of the proof of \cite[Theorem 4.1]{HR24+}. We put
\begin{align}
\mathcal{Q}:=\{\ell\nmid 2N_E\ |\ \ell\equiv5\bmod 8,\ \widetilde{E}_{\ell}(\F_{\ell})[2]\neq\{O\}\}.
\end{align}
Let $L:=\Q(E[2])\Q(\zeta_8)$ and $G:=\gal(L/\Q)$. By Lemma \ref{Lem:linear_disjoint}, we have 
\begin{align}
G\cong \gal(\Q(E[2])/\Q)\times(\Z/8\Z)^{\times}.
\end{align}
The set $\mathcal{Q}$ corresponds to the subset
\begin{align}
A\times \{5\}\subset G,
\end{align}
where
\begin{align}
A:=\{\sigma\in \gal(\Q(E[2])/\Q)\ |\ E[2]^{\sigma}\neq\{O\}\}.
\end{align}
By the Chebotarev density theorem, $\mathcal{Q}$ has positive density. Each element $\ell$ of $\mathcal{Q}$ satisfies $n_{\ell}=0$. Since $\lambda_2^*(E,\omega^i)\equiv N^*\bmod 2$, we can write $N^*=\lambda^*_2(E,\omega^i)+2r$ for some $r\in \Z_{\geq0}$. We choose $r$ distinct primes $\ell_1,\ldots,\ell_r$ in $\mathcal{Q}$ and set $D_1:=\ell_1\cdots \ell_r$.  We define
\begin{align}
\mho_E':=\mho_E\cup \{\ \ell\text{ is a rational prime }|\ \ell\mid 2N_ED_1\}.
\end{align}
By the Chebotarev density theorem for $\Q(E[2])\Q(i)/\Q$, there exists a rational prime $q_0$ such that 
\begin{align}
q_0\equiv 3\bmod 4,\quad q_0\notin \mho_E'.
\end{align}
Let $q_1,\ldots,q_s\notin\mho_E'$ be distinct primes such that $D_2:=q_1\cdots q_s\equiv 1\bmod 4$ and we put $D:=D_1D_2$. By the Matsuno-type formula \eqref{eq:Matsuno-type_formula}, we have
\begin{align}
\lambda_2^*(E^D,\omega^i)=\lambda_2^*(E,\omega^i)+\sum_{\substack{\ell\mid D\\ 2\mid \sharp\widetilde{E}_{\ell}(\F_{\ell})}}2=\lambda_2^*(E,\omega^i)+2r=N^*.
\end{align}
We define
\begin{align}
N_{\mho_E'}^+\left(X/D_1\right):=\left\{D_2\leq \frac{X}{D_1}\ \middle|\ D_2=q_1\ldots q_s\equiv1\bmod 4\text{ for }s\in\Z_{\geq0},\ q_j\notin\mho_E',\ D_2\text{ is square-free}\right\}.
\end{align}
Then
\begin{align}
m_{E,N^*}^*(X)\geq \sharp N_{\mho_E'}^+\left(X/D_1\right)=: n_{\mho_E'}^+\left(X/D_1\right).
\end{align}
We define the set $\mathcal{B}$ by
\begin{align}
\mathcal{B}:=\left\{d\leq \frac{X}{q_0D_1}\ \middle|\ d\text{ is square-free},\ d=q_1\cdots q_s\text{ for }q_j\notin \mho_E',\ q_0\nmid d\right\},
\end{align}
and  the map $j:\mathcal{B}\to N_{\mho_E'}^+(X/D_1)$, for each $d\in \mathcal{B}$,
\begin{align}
j(d):=\begin{cases}
d&(d\equiv1\bmod 4),\\
q_0d&(d\equiv3\bmod 4).
\end{cases}
\end{align}
Since $j$ is injective and 
\begin{align}
\sharp \mathcal{B}\geq n_{\mho_E'}(X/q_0D_1)-n_{\mho_E'}(X/q_0^2D_1),
\end{align}
we obtain
\begin{align}
m_{E,N^*}^*(X)&\geq n_{\mho_E'}^+\left(X/D_1\right)\\
&\geq \sharp \mathcal{B}\\
&\geq n_{\mho_E'}(X/q_0D_1)-n_{\mho_E'}(X/q_0^2D_1).
\end{align}
From Proposition \ref{Lem:density_omega_E}, there exists a constant $c>0$ such that
\begin{align}
n_{\mho_E'}(X/q_0D_1)-n_{\mho_E'}(X/q_0^2D_1)\sim c\left(\frac{1}{q_0D_1}-\frac{1}{q_0^2D_1}\right)\frac{X}{(\log X)^{\frac{2}{3}}}.
\end{align}
Therefore we obtain
\begin{align}
m_{E,N^*}^*(X)\geq n_{\mho_E'}^+\left(X/D_1\right)\gg \frac{X}{(\log X)^{\frac{2}{3}}}.
\end{align}
\end{proof}
%%%%%%%%%%%%%%%%%%%%%%

%%%%%%%%%%%%%%%%%%%%%%
\section{Examples}\label{sec:example}
%%%%%%%%%%%%%%%%%%%%%%
In this section, we give examples of Theorem \ref{Thm:Matsuno_formula} and Theorem \ref{Thm:distribution_lambda}. We consider sharp/flat Iwasawa invariants only for the trivial tame character. For simplicity, we denote their Iwasawa invariants by $\mu_2^{\sharp/\flat}(A)$ and $\lambda_2^{\sharp/\flat}(A)$ for each elliptic curve $A$ defined over $\Q$ with good supersingular reduction at 2. We use SageMath \cite{Sage} and LMFDB \cite{lmfdb}. We compute the $\sharp/\flat$ Iwasawa invariants using SageMath and Pollack's Sage code \cite{PolCode} \texttt{IwInv.sage}, available from the repository \texttt{rpollack9974/Iwasawa-invariants}.  More precisely, we used the
function \texttt{iwasawa\_invariants\_of\_ec(E,p)}. 
\begin{rem}
Pollack's code computes the
$\mu^{\pm}$- and $\lambda^{\pm}$-invariants, studied by Perrin-Riou \cite{Per03} and Kurihara \cite{Kur02}, attached to the Mazur--Tate
queue sequence.  In the examples below, the computed values satisfy
$\mu^+=\mu^-=0$.  Hence, taking into account the $p=2$ sign convention explained in  \cite[Definition 8.8, footnote 6]{Spr17}, \cite[Corollary 8.9]{Spr17} shows that these invariants
coincide with the Iwasawa invariants of Sprung's sharp/flat
$2$-adic $L$-functions. With our convention $+=\flat$ and
$-=\sharp$, we denote them by $\mu_2^{\sharp/\flat}$ and
$\lambda_2^{\sharp/\flat}$.
\end{rem}
%%%%%%%%%%%%%%%%%%%%%%
\subsection{Examples of Theorem \ref{Thm:Matsuno_formula}}\label{ss:Matsuno-type_formula_ex}
%%%%%%%%%%%%%%%%%%%%%%

\begin{exa}
Let $E_3$ be the elliptic curve over $\Q$ with LMFDB label 101.a1 defined by
\begin{align}
y^2+y=x^3+x^2-x-1.
\end{align}
The elliptic curve $E_3$ has good supersingular reduction at 2 and the Frobenius trace of $E_3$ at 2 is equal to 0. 
Using Pollack's algorithm \cite{PolCode}, we obtain
\begin{align}
(\mu_2^{\sharp}(E_3),\mu_2^{\flat}(E_3))=(0,0),\quad (\lambda_2^{\sharp}(E_3),\lambda_2^{\flat}(E_3))=(1,4).
\end{align}
The conductor $N_{E_3}$ is equal to $101$. We consider the case $D=5$. Since $\sharp\widetilde{(E_3)}_5(\F_5)=1+5+1=7$ by LMFDB, the Matsuno-type formula gives
\begin{align}
\mu_2^{\sharp/\flat}(E_3^5)=0,\quad \lambda_2^{\sharp/\flat}(E_3^5)=\lambda_2^{\sharp/\flat}(E_3).
\end{align}
This agrees with the output of Pollack's algorithm:
\begin{align}
(\mu_2^{\sharp}(E_3^5),\mu_2^{\flat}(E_3^5))=(0,0),\quad (\lambda_2^{\sharp}(E_3^5),\lambda_2^{\flat}(E_3^5))=(1,4).
\end{align}
Next, we consider the case $D=21=3\cdot 7$. According to LMFDB, 
\begin{align}
\sharp\widetilde{(E_3)}_3(\F_3)=4+2=6,\quad \sharp\widetilde{(E_3)}_7(\F_7)=8+2=10.
\end{align}
Then, by the Matsuno-type formula, we obtain
\begin{align}
\mu_2^{\sharp/\flat}(E_3^{21})=0,\quad \lambda_2^{\sharp/\flat}(E_3^{21})=\lambda_2^{\sharp/\flat}(E_3)+2^{n_3+1}+2^{n_7+1}=\lambda_2^{\sharp/\flat}(E_3)+6.
\end{align}
This agrees with the output of Pollack's algorithm:
\begin{align}
(\mu_2^{\sharp}(E_3^{21}),\mu_2^{\flat}(E_3^{21}))=(0,0),\quad (\lambda_2^{\sharp}(E_3^{21}),\lambda_2^{\flat}(E_3^{21}))=(7,10).
\end{align}
\end{exa}

\begin{exa}
Let $E_4$ be the elliptic curve over $\Q$ with LMFDB label 67.a1 defined by
\begin{align}
y^2+y=x^3+x^2-12x-21.
\end{align}
The elliptic curve $E_4$ has good supersingular reduction at 2 and the Frobenius trace of $E_4$ at 2 is equal to 2. Using Pollack's algorithm \cite{PolCode}, we obtain
\begin{align}
(\mu_2^{\sharp}(E_4),\mu_2^{\flat}(E_4))=(0,0),\quad (\lambda_2^{\sharp}(E_4),\lambda_2^{\flat}(E_4))=(1,0).
\end{align}
The conductor $N_{E_4}$ is equal to $67$. We consider the case $D=33=3\cdot 11\ (\equiv 1\bmod 8)$. By LMFDB, we have
\begin{align}
\sharp\widetilde{(E_4)}_3(\F_3)=4+2=6,\quad \sharp\widetilde{(E_4)}_{11}(\F_{11})=12+4=16.
\end{align}
Then, by the Matsuno-type formula, we obtain
\begin{align}
\mu_2^{\sharp/\flat}(E_4^{33})=0,\quad \lambda_2^{\sharp/\flat}(E_4^{33})=\lambda_2^{\sharp/\flat}(E_4)+2^{n_3+1}+2^{n_{11}+1}=\lambda_2^{\sharp/\flat}(E_4)+4.
\end{align}
This agrees with the output of Pollack's algorithm:
\begin{align}
(\mu_2^{\sharp}(E_4^{33}),\mu_2^{\flat}(E_4^{33}))=(0,0),\quad (\lambda_2^{\sharp}(E_4^{33}),\lambda_2^{\flat}(E_4^{33}))=(5,4).
\end{align}
\end{exa}

\begin{exa}
Let $E_5$ be the elliptic curve over $\Q$ with LMFDB label 37.a1 defined by
\begin{align}
y^2+y=x^3-x.
\end{align}
The elliptic curve $E_5$ has good supersingular reduction at 2 and the Frobenius trace of $E_5$ at 2 is equal to $-2$. Using Pollack's algorithm \cite{PolCode}, we obtain
\begin{align}
(\mu_2^{\sharp}(E_5),\mu_2^{\flat}(E_5))=(0,0),\quad (\lambda_2^{\sharp}(E_5),\lambda_2^{\flat}(E_5))=(1,2).
\end{align}
The conductor $N_{E_5}$ is equal to $37$. We consider the case $D=13\ (\equiv 5\bmod 8)$. By LMFDB, we have
\begin{align}
\sharp\widetilde{(E_5)}_{13}(\F_{13})=14+2=16.
\end{align}
Then, by the Matsuno-type formula, we obtain
\begin{align}
\mu_2^{\sharp/\flat}(E_5^{13})=0,\quad \lambda_2^{\sharp/\flat}(E_5^{13})=\lambda_2^{\sharp/\flat}(E_5)+2^{n_{13}+1}=\lambda_2^{\sharp/\flat}(E_5)+2.
\end{align}
This agrees with the output of Pollack's algorithm:
\begin{align}
(\mu_2^{\sharp}(E_5^{13}),\mu_2^{\flat}(E_5^{13}))=(0,0),\quad (\lambda_2^{\sharp}(E_5^{13}),\lambda_2^{\flat}(E_5^{13}))=(3,4).
\end{align}

\end{exa}

\subsection{Examples of Theorem \ref{Thm:distribution_lambda}}
We consider three elliptic curves $E_i$ for $i\in\{3,4,5\}$ defined in Subsection \ref{ss:Matsuno-type_formula_ex}. For each $*\in\{\sharp,\flat\}$ and $i\in\{3,4,5\}$, let $N_i^*$ be any integer such that 
\begin{align}
N_i^*\geq\lambda_2^*(E_i),\quad N_i^*\equiv \lambda_2^*(E_i)\bmod 2.
\end{align}
Using Pollack's algorithm and LMFDB, we find that $\mu_2^*(E_i)=0$ and $\Imag(\bar{\rho}_{E_i,2})$ for $i\in\{3,4,5\}$ is generated by
\begin{align}
\begin{pmatrix}1&1\\ 1&0\end{pmatrix},\quad \begin{pmatrix}1&0\\ 1&1\end{pmatrix}.
\end{align}
Hence $\bar{\rho}_{E_i,2}$ is surjective. Therefore, Theorem \ref{Thm:distribution_lambda} gives
\begin{align}
m_{E_i,N_i^*}^*(X)\gg \frac{X}{(\log X)^{\frac{2}{3}}}
\end{align}
for each $i\in\{3,4,5\}$.

\section*{Methods}
During the preparation of this manuscript, the author used ChatGPT solely as a supplementary tool to identify possible inconsistencies. All mathematical arguments and conclusions were independently formulated and verified by the author, who takes full responsibility for the content of the manuscript.

\section*{Data availability}
The elliptic-curve data were obtained from the LMFDB \cite{lmfdb}. The period ratios in Subsection~\ref{ss:Period_twists} were computed using SageMath \cite{Sage}, and the sharp/flat Iwasawa invariants in Section~\ref{sec:example} were computed using Pollack's SageMath code \cite{PolCode}. Additional computational scripts are available at \url{https://github.com/TaigaAdachi/Sharp_flat_2-adic_Iwasawa_invariants_quadratic_twists_elliptic_curves}.

{
\bibliographystyle{plain}
\bibliography{biblio}
}

\end{document}